\documentclass[12pt]{amsart}
\usepackage[numbers, square]{natbib}
\usepackage{amssymb}
\usepackage{amsmath}
\usepackage{amsfonts}
\usepackage{geometry}
\usepackage{graphicx}
\usepackage{amsthm}
\usepackage{hyperref}
\usepackage[dvipsnames]{xcolor}
\usepackage{esint}

\providecommand{\U}[1]{\protect\rule{.1in}{.1in}}
\theoremstyle{plain}

\newtheorem{corollary}{Corollary}

\newtheorem{lemma}{Lemma}

\newtheorem{proposition}{Proposition}

\newtheorem{theorem}{Theorem}
\theoremstyle{remark}
\newtheorem{remark}{Remark}
\numberwithin{equation}{section}

\newcommand{\norm}[1]{\left\vert\left\vert #1\right\vert\right\vert}

\newcommand{\R}{\ensuremath{\mathbb{R}}}
\newcommand{\Z}{\ensuremath{\mathbb{Z}}}

\begin{document}

\title[]
{Spectral inequality for Schr\"odinger equations with power growth potentials }
\author{ Jiuyi Zhu}
\address{
Department of Mathematics\\
Louisiana State University\\
Baton Rouge, LA 70803, USA\\
Email:  zhu@math.lsu.edu }
\author {Jinping Zhuge}
\address{
 Morningside Center of Mathematics\\
the Academy of Mathematics and Systems Science\\ Chinese Academy of Sciences\\
Beijing, China\\
Email: jpzhuge@amss.ac.cn }

\begin{abstract}
We prove a spectral inequality for Schr\"odinger equations with power growth potentials, which particularly confirms a conjecture in \cite{DSV}. This  spectral inequality depends on the decaying density of the sensor sets, and the growth rate of potentials. The proof relies on three-ball inequalities derived from modified versions of quantitative global and local Carleman estimates that take advantage of the gradient information of the potentials.
\end{abstract}

\maketitle
\section{Introduction}

A spectral inequality in control theory for a nonnegative self-adjoint elliptic operator $H$ takes the following form
\begin{align}
 \|\phi\|_{L^2(\mathcal{M})}\leq C_0 e^{C_1 \lambda^t}  \|\phi\|_{L^2(\Omega)}, \quad \text{for any } \phi\in  Ran(P_\lambda(H)),
 \label{spec-in}
\end{align}
with some universal constants $t\in (0,1), C_0>0, C_1>0$, where $(\mathcal{M},g)$ is a Riemannian manifold, $\Omega \subset \mathcal{M}$ is a measurable subset, and $P_\lambda(H) = \chi_{(-\infty, \lambda)}(H)$ is the resolution of identity associated to $H$.
If $\mathcal{M}$ is a smooth compact manifold and $H=-\triangle_g$ is the negative Laplace-Beltrami operator on  $\mathcal{M}$, then $P_\lambda(H)$ is the spectral projection and $Ran(P_\lambda(H))$ consists of finite sums of Laplace eigenfunctions, i.e., $\phi =\sum_{\lambda_k\leq \lambda} \phi_{k}$, where
\begin{equation}
-\triangle_g \phi_{k} = \lambda_k\phi_{k} \quad \mbox{on} \ \mathcal{M}.
\label{classs}
\end{equation}
In this case, the spectral inquality \eqref{spec-in} holds with the sharp exponent $t = \frac12$, which originated from \cite{LR}, \cite{JL} and \cite{LZ}. This result was used to study the null-controllability problem for the corresponding heat equation in \cite{LR}, the Hausdorff measure of nodal sets for finite sums of eigenfunctions in \cite{JL}, and the null-controllability of thermoelasticity system in \cite{LZ}. 
See the review article \cite{LL} for more extensive literature on the theory and applications in this situation.

The spectral inequality (\ref{spec-in}) is also closely related to the quantitative unique continuation property for the second order elliptic equations. For example, the following sharp doubling inequality
\begin{equation}\|\phi_{k}\|_{L^2(\mathbb
B_{2r}(x))}\leq  e^{C{\sqrt{\lambda_k}}}\|\phi_{k}\|_{L^2(\mathbb
B_{r}(x))}, \quad \text{for all } \mathbb B_{2r}(x)\subset \mathcal{M},
\label{doub}
\end{equation}
was obtained in \cite{DF}  for Laplace eigenfunctions in (\ref{classs}), where $C$ depends only on $\mathcal{M}$. The common philosophy for the spectral inequality (\ref{spec-in}) and the doubling inequality (\ref{doub}) is that
 the information of eigenfunctions at any local area quantitatively controls the growth information in a larger region. We mention that the quantitative unique continuation is closely connected to the Landis' conjecture on the sharp decay rate at infinity \cite{KL,KSW,Zh1,LMNN}, the uncertainty principles and Logvinenko-Sereda theorem \cite{DZ,BJP, K}.

In the present paper, we study the spectral inequality for the Schr\"odinger operator $H=-\triangle +V(x)$ in $\mathcal{M} = \R^n$. If
$\lim_{|x|\to \infty} V(x)=+\infty$, the inverse operator $H^{-1}$ is compact in $L^2(\R^n)$ and therefore the spectrum of $H$ are discrete (countable) and unbounded, consisting of positive eigenvalues with finite multiplicity that diverges to infinity. Moreover,
for $\phi\in Ran(P_\lambda(H))$,
\begin{align}
\phi=\sum_{\lambda_k\leq \lambda} \alpha_k \phi_k, \quad \mbox{with}\ \ \alpha_k=\langle \phi_k, \phi\rangle,
\label{phi-s}
\end{align}
where $\phi_k$ is the eigenfunction of $H$ corresponding to the eigenvalue $\lambda_k$, namely,
\begin{align}
-\triangle \phi_k +V(x)\phi_k=\lambda_k \phi_k \quad \text{ in } \mathbb{R}^n,
\label{eigen-k}
\end{align}
and $\{ \phi_k: \lambda_k \le \lambda \}$ forms an orthogonal basis of $ Ran(P_\lambda(H))$. Here and after, we will always assume that $\lambda$ is no less than the smallest eigenvalue $\lambda_1>0$.

A measurable \emph{sensor set} $\Omega \subset \R^n$ is called \emph{efficient} if the spectral inequality \eqref{spec-in} holds with some $t\in (0,1)$. An efficient sensor set would imply the null-controllability for the corresponding heat equation in $\R^n$; see, e.g., \cite[Theorem 2.8]{NTTV1}. It has been shown in \cite[Example 2.5]{DSV1} that in general the efficient set $\Omega$ cannot be bounded. However, with the power growth of the potential $V$, the eigenfunctions are well localized and decaying exponentially. This allows that the volume of $\Omega$ can be bounded and $\Omega$ can be very sparse far away from the origin. In particular, \eqref{spec-in} was proved in \cite{DSV1} for harmonic oscillator $H = -\triangle + |x|^2$ with $t = \frac{\sigma}{2} + \frac12$, provided $\Omega$ satisfying
\begin{equation}\label{meas-Omega}
    \big|\Omega \cap \big( k + (-\frac{d}{2}, \frac{d}{2})^n \big) \big| \ge \gamma^{1+|k|^\sigma} d^n,
\end{equation}
for each $k\in \Z^n$ with some $d>0, \gamma \in (0,1)$ and $\sigma \in [0,1)$. The special case $\sigma =0$, corresponding to the so-call thick sets, has been used previously to study spectral inequality in, e.g., \cite{MP,BJP}. For general potentials $V$ that are merely bounded or have mild local singularities, the spectral inequalities have been established in \cite{NTTV, DSV} under a stronger condition on the sensor set $\Omega$, which will be introduced below.

Denote by $\Lambda_L=(-\frac{L}{2}, \frac{L}{2})^n$ the cube with side length $L>0$. Let $d>0$, $\gamma \in (0,1)$ and $\sigma\in [0,1)$. 
We consider the measurable sensor sets $\Omega\subset \mathbb R^n$ satisfying the following property: there exists an \emph{equidistributed sequence} $\{ z_k: k\in \Z^d \} \subset \R^n$ such that
\begin{align}
\Omega\cap (k+\Lambda_d) \supset  \mathcal B_{\gamma^{1+|k|^\sigma}{d}} (z_k)
\label{geom-om}
\end{align}
for all $k\in \mathbb{Z}^n$.
A simple example of the sensor sets is
 \begin{align}
 \Omega=\bigcup_{k\in \mathbb Z^n} \mathcal{B}_{2^{-(1+|k|^\sigma)}}(k).
 \label{example}
 \end{align}
Clearly, $\Omega$ in (\ref{example}) satisfies (\ref{geom-om}) with $\gamma=\frac{1}{2}$ and $d=1$ and has finite volume in the case $\sigma>0$.
If $\sigma=0$, the sensor set $\Omega$ satisfying the assumption (\ref{geom-om}) is called a $(d, \gamma)$-equidistributed set, which was introduced in
\cite{RV} and studied also in, e.g., \cite{NTTV}.
In this paper, for simplicity and without loss of generality, we will fix $d=1$.

Among the growing nonnegative potentials, the simplest ones are $V(x) = |x|^\beta$ with $\beta>0$. The following conjecture was made in \cite[Conjecture 1.5]{DSV}.

\textbf{Conjecture (I)}: Let $H=-\triangle +|x|^\beta$ for some $\beta>0$.  Assume that $\Omega$ satisfies (\ref{geom-om}) with $d=1, \sigma \in [0,\infty)$ and $\gamma\in (0, \frac{1}{2})$.  Then there exists a constant $C$, depending only on $\beta, \sigma$ and $n$, such that
\begin{align}\label{ineq.conj}
\|\phi\|_{L^2(\mathbb R^n)}\leq (\frac{1}{\gamma})^{C\lambda^{\frac{\sigma}{\beta}+\frac{1}{2}}} \|\phi\|_{L^2(\Omega)}, \quad \text{for all } \phi \in Ran(P_\lambda(H)).
\end{align}

If $\beta=2$, then $H=-\triangle +|x|^2$ is the standard harmonic oscillator. In this case, as mentioned earlier, the above Conjecture (I) has been proved in \cite{DSV1} under the condition \eqref{meas-Omega} by using the fact that the eigenfunctions of harmonic oscillator can be written in terms of explicit Hermite functions. For $\beta \neq 2$, a spectral inequality was proved in \cite{DSV} with suboptimal exponent $\frac{\sigma}{\beta} + \frac23$ instead of the conjectured exponent $\frac{\sigma}{\beta}+\frac{1}{2}$ in \eqref{ineq.conj}.
In this paper, we prove a spectral inequality with conditioned potential $V$, which solves Conjecture (I) as a corollary.

 Now, we state our assumption on $V$ and the main theorem.
 
\textbf{Assumption (A)}: Assume $V\in L^\infty_{\rm loc}(\R^n)$ satisfies the following two conditions: 
\begin{itemize}
    \item There exist positive constants $c_1$ and $\beta_1$ such that for all $x\in \R^n$,
    \begin{equation}\label{V.growth}
    c_1(|x| - 1)_+^{\beta_1} \le V(x),
    \end{equation}
    where $(a)_+ = \max\{a,0 \}$.
    \item We can write $V = V_1 + V_2$ such that there exist positive constants $c_2$ and $\beta_2$ such that
    \begin{equation}\label{V.bound}
    |V_1(x)| + |D V_1(x)| + |V_2(x)|^{\frac43} \le c_2(|x| + 1)^{\beta_2}.
\end{equation}
\end{itemize}
Throughout, we will use the notations $D_i = \frac{\partial }{\partial x_i}$ for partial derivatives and $D = (D_1,D_2,\cdots)$ for gradients. Roughly speaking, $V_1$ encloses the regular principle part of the potential $V$, while $V_2$ can be viewed as a rough small perturbation (so $V$ does not need to be differentiable). This flexible assumption is sufficient for many applications, particularly including $V(x) = |x|^\beta$ for any $\beta>0$. 

\begin{theorem}
Let $H=-\triangle +V(x)$. Assume  that $V$ satisfies Assumption (A) and $\Omega$ satisfies (\ref{geom-om}) with $d=1, \sigma\in [0,\infty)$ and $\gamma\in (0, \frac{1}{2})$. Then there exists a constant $C$ depending only on $\beta_1, \beta_2, c_1, c_2, \sigma$ and $n$ such that
\begin{align}
\|\phi\|_{L^2(\mathbb R^n)}\leq (\frac{1}{\gamma})^{C\lambda^{\frac{\sigma}{\beta_1}+\frac{\beta_2}{2\beta_1}}} \|\phi\|_{L^2(\Omega)},\quad \text{for all } \phi\in Ran(P_\lambda(H)).
\label{aim-res}
\end{align}
\label{th1}
\end{theorem}

In Theorem \ref{th1}, the sensor set $\Omega$ is efficient, if $t=\frac{\sigma}{\beta_1}+\frac{\beta_2}{2\beta_1}<1$, i.e., $\sigma<\beta_1-\frac{\beta_2}{2}$. If $\beta_1=\beta_2=\beta$, then $\sigma<\frac{\beta}{2}$. In this case, Theorem \ref{th1} leads to the observability of the corresponding heat equations as in \cite[Corollary 1.5]{DSV}.
We point out that under a weaker condition on $DV$, a similar spectral inequality was shown in \cite{DSV} with suboptimal exponent $\frac{\sigma}{\beta_1} + \frac{2\beta_2}{3\beta_1}$. The factor $\frac23$, also appearing in \cite{NTTV1,NTTV2}, seems to be a natural barrier if the potential is merely bounded measurable. In Theorem \ref{th1} we improve the exponent to the expected one $\frac{\sigma}{\beta_1} + \frac{\beta_2}{2\beta_1}$, which is consistent with the sharp result for harmonic oscillator (i.e., $\beta_1 = \beta_2 = 2$) in \cite{DSV1} and seems to be optimal in general. Our proof largely follows the strategy of \cite{LR,JL,NTTV}. The new ingredient in the argument is that we apply certain quantitative Carleman estimates to take advantage of the growth condition on $DV_1$. These quantitative Carleman estimates are initially motivated by \cite{DF,Lau, Zh}. Especially, a new version of quantitative global Carleman estimate was shown  by the first author of this paper in \cite{Zh2} for the study of boundary doubling inequalities for Laplace eigenfunctions, and is generalized in this paper to study the spectral inequality.

As a straightforward corollary of Theorem \ref{th1}, we have
\begin{corollary}
Conjecture (I) is true.
\label{th2}
\end{corollary}
\begin{proof}
For $\beta\ge 1$, Theorem \ref{th2} follows from Theorem \ref{th1} with $V_1 = |x|^\beta$ and $V_2 = 0$. For $0<\beta<1$, we write $|x|^\beta = |x|^\beta (1-\eta(x)) + |x|^\beta \eta(x)$, where $\eta \in C_0^\infty(\R^n)$ is a smooth cut-off function such that $\eta = 1$ in $\mathcal{B}_1$ and $\eta = 0$ in $\R^n\setminus \mathcal{B}_2$. Then Theorem \ref{th2} follows from Theorem \ref{th1} with $V_1 = |x|^\beta (1-\eta(x))$ and $V_2 = |x|^\beta \eta(x)$.
\end{proof}

Our paper is mostly self-contained except for the proofs of Proposition \ref{pro1} (a Carleman estimate) and Proposition \ref{propo-new} (a localization property). The proof of Proposition \ref{pro1}, which is in the same spirit of the proof of Proposition \ref{Pro4}, can be extracted from  \cite{Zh}. Proposition \ref{propo-new} is taken from \cite{DSV} with a simple improvement; see Remark \ref{rmk.Decay}.
Throughout the paper, we use $\mathcal{B}_r$ and $\mathbb{B}_r$ to represent balls in $\R^n$ and $\R^{n+1}$ respectively, and use the notation $\mathbb{B}_r^+ = \mathbb{B}_r \cap \{ x_{n+1}>0 \}$.
The letters $C$, $\hat{C}$, or $C_i$ denote positive constants that do not depend on $\lambda$, and may vary from line to line. We use $a\approx b$ to denote $cb \le a\le Cb$ with universal positive constants $c$ and $C$ depending only on $n$.

The organization of the paper is as follows. In section 2, we obtain two types of three-ball inequalities, using quantitative global and local Carleman estimates. Section 3 is devoted to the proof of the spectral inequality in Theorem
 \ref{th1}. The Appendix in section 4 includes the proof of the quantitative global Carleman estimate and a lemma on the geometric properties of some introduced sets. 

\textbf{Acknowledgements.} The first author is partially supported by NSF DMS-2154506. The second author is supported by grants from NSFC and AMSS-CAS.

\section{Three-ball inequalities }
In this section, we prove two types of three-ball inequalities for the solutions of
\begin{align}
-\triangle v+V v=0 \quad \mbox{in} \ {\mathbb R}^{n+1}.
\label{eqn-ste}
\end{align}
Our main tools to obtain the three-ball inequalities are the quantitative Carleman estimates.
Carleman estimates are weighted integral inequalities with weight functions in the form of $e^{\tau\psi}$, where $\psi$ usually satisfies some convexity condition and the Carleman parameter $\tau$ is assumed to be large. To obtain the quantitative Carleman estimates, we need a lower bound for the Carleman parameter $\tau$.

First we present a quantitative global Carleman estimate and then prove a three-ball inequality involving one ball on the boundary. We choose the weight function 
$$\psi(\hat{r})= e^{-s\hat{r}(y)} $$ 
with $\hat{r}=\hat{r}(y)=|y-b|$ where $b=(0, \cdots, 0, -b_{n+1})$ with some small positive $b_{n+1}>0$. This modified distance function does not generate singularity in the upper half-space when taken derivatives. Let $\delta\in (0,\frac12)$ be a scale parameter. In this paper, we fix $b_{n+1} = \frac{\delta}{100}$.

\begin{proposition}[First Carleman estimate]
Let $\delta\in (0,\frac12)$. There exist positive constants $C_0, C_1$ and $C_2$ depending on $n$ such that for any  $u\in C^\infty_0(\mathbb B_\delta^+\cup\{y_{n+1}=0\})$, $s = \frac{C_2}{\delta}$ and
\begin{align*}\tau\geq C_1\big( 1+\|V_1\|^{\frac{1}{2}}_{W^{1,\infty}(\mathbb B_{\delta}^+)} + \norm{V_2}^{\frac23}_{L^\infty(\mathbb B_{\delta}^+)} \big) ,\end{align*}
one has
\begin{align}
\| e^{\tau \psi}(- \triangle u +Vu)\|_{L^2(\mathbb B_{\delta}^+)} & + \tau^\frac{1}{2} s \| \psi^{\frac{1}{2}} e^{\tau \psi}\frac{\partial u}{\partial y_{n+1} }\|_{L^2(\partial \mathbb B_{\delta}^+\cap\{y_{n+1}=0\})} \nonumber \\
&\geq C_0 \tau^\frac{3}{2}s^2 \|\psi^{\frac{3}{2}} e^{\tau \psi} u\|_{L^2( \mathbb B_{\delta}^+)}+ C_0 \tau^\frac{1}{2} s \| \psi^{\frac{1}{2}}  e^{\tau \psi} D u \|_{L^2( \mathbb B_{\delta}^+)}.
 \label{useew}
\end{align}
\label{Pro4}
\end{proposition}

The above proposition is a modified version of \cite[Proposition 2]{Zh2}. The proof is given in the Appendix.

Define the sets
\begin{align}
W_1=\{y\in \mathbb R^{n+1}_+| \ |y-b|\leq \frac{1}{4}\delta\} \nonumber,
\end{align}
\begin{align}
W_2=\{y\in \mathbb R^{n+1}_+| \ |y-b|\leq \frac{1}{2}\delta\} \nonumber,
\end{align}
\begin{align}
W_3=\{y\in \mathbb R^{n+1}_+| \ |y-b|\leq \frac{2}{3}\delta\} \nonumber,
\end{align}
where $0<\delta<\frac{1}{2}$. Note that $W_1\subset W_2\subset W_3\subset \mathbb B_\delta\subset \mathbb R^{n+1}$.

 Let $L$ be an odd integer to be determined in the proof of Theorem \ref{th1} in Section 3.
Let $z_i\in\mathbb R^n, i\in \Z^n$, be a $(1, \delta)$-equidistributed sequence in $\Lambda_L = (-\frac{L}{2},\frac{L}{2})^n$, namely, $\mathcal{B}_\delta(z_k) \subset k+\Lambda_1$ for each $k\in Q_L:= \Lambda_L \cap \Z^n$.
Let $W_j(z_i)=W_j+(z_i, 0)$ be the translation of the set $W_j\subset \mathbb R^{n+1}_+$ for $j=1, 2, 3$.
We denote by $\mathcal{B}_\delta(z_i)$ the ball with radius $\delta$ centered at $z_i$ in $\mathbb R^n$, and introduce
\begin{align*} 
P_j(L) =\bigcup_{i\in Q_L} W_j(z_i) \quad \mbox{and} \quad D_\delta(L) = \bigcup_{i\in Q_L} {\mathcal{B}}_\delta(z_i).
\end{align*}
With the aid of Proposition \ref{Pro4},
 we are able to prove the following global three-ball inequality. See e.g., \cite{ARRV, Lin} for some variants of this three-ball inequality.
\begin{lemma}
Let $\delta \in (0,\frac12)$. Let $v$ be the solution of (\ref{eqn-ste}) with $v(y)=0$ on the hyperplane $\{y|y_{n+1}=0\}$. There exist $0<\alpha<1$ and $C>0$, depending only on $n$, such that
\begin{align}
\|v\|_{H^1(P_1(L))}\leq \delta^{-\alpha} \exp\big(C(1+\mathcal{G}(V_1,V_2,9\sqrt{n}L) \big) \|v\|^\alpha_{H^1(P_3(L))} \|\frac{\partial v}{\partial y_{n+1}}\|^{1-\alpha}_{L^2(D_\delta(L))},
\label{threeball-1}
\end{align}
where
\begin{equation}\label{defG}
    \mathcal{G}(V_1,V_2,L) = \|V_1\|^{\frac{1}{2}}_{W^{1,\infty}(\Lambda_L)} + \norm{V_2}^{\frac23}_{L^\infty(\Lambda_L)}.
\end{equation}
\label{lemma-2}
\end{lemma}
\begin{proof}
We choose a cut-off function $\eta\in C^\infty_0(\mathbb B_\delta)$ such that $\eta(y)=1$ for $y\in W_2$ and $\eta(y)=0$ outside $W_3$. Then we have $\|\triangle \eta\|_{L^\infty(\mathbb{B}_\delta^+)}\leq \frac{C}{\delta^2}$ and  $\|D \eta\|_{L^\infty(\mathbb{B}_\delta^+)}\leq \frac{C}{\delta}$.
We substitute $u=\eta v$ into the Carleman estimates (\ref{useew}) to derive that
\begin{align}
\int_{W_3} e^{2\tau \psi}\big (-\triangle (\eta v)+V \eta v\big)^2+& \tau s^2\int_{\mathcal{B}_\delta}\psi  e^{2\tau \psi} (\frac{\partial (v\eta)}{\partial y_{n+1}})^2 \nonumber \\&\geq C_0\tau s^2 \int_{W_3} \psi^3  e^{2\tau \psi} \big( \eta^2 v^2+ | D(\eta v)|^2\big),
\label{subs-star}
\end{align}
where  we have used the fact that $|\psi|<1$.
From the equation $-\triangle v+V(y) v=0$, we obtain
\begin{align}
-\triangle (\eta v)+V \eta v =-\triangle \eta v+ 2D \eta\cdot D v.
\end{align}
Since $\psi$ is  non-increasing and $D \eta$ is supported in $W_3\setminus W_2$, we have
\begin{align}
\int_{W_3} e^{2\tau \psi}\big (-\triangle (\eta v)+V\eta v\big)^2& \leq 2 \int_{W_3} e^{2\tau \psi} (|\triangle \eta v|^2+  4 |D \eta |^2 |D v|^2) \nonumber \\
&\leq \frac{C}{\delta^4} e^{2\tau \psi(\frac{\delta}{2})} \|v\|^2_{H^1(W_3)}.
\label{three-1-1}
\end{align}
Since $v(y)=0$ on $\mathcal{B}_\delta$  (i.e. $\mathcal{B}_\delta=\mathbb B_\delta\cap \{y\in \mathbb R^{n+1}|y_{n+1}=0\}$)  and $s= \frac{C_2}{\delta}$, it holds that
\begin{align}
\tau s^2 \int_{ {\mathcal{B}}_\delta}\psi  e^{2\tau \psi} (\frac{\partial (v\eta)}{\partial y_{n+1}})^2 &\leq \tau s^2 \int_{ { \mathcal{B}}_\delta} \psi  e^{2\tau \psi} | \frac{\partial v}{\partial y_{n+1}} \eta+ \frac{\partial \eta}{\partial y_{n+1}} v|^2  \nonumber \\
&\leq \frac{C}{\delta^2}e^{2\tau} \int_{\mathcal{B}_\delta}|\frac{\partial v}{\partial y_{n+1}}|^2,
\label{three-2-1}
\end{align}
where we have also used the fact that $|\psi|\leq 1$.
Now we consider the right hand side of (\ref{subs-star}). Since $W_1\subset W_3$, from the property of the cut-off function $\eta$, we have
\begin{align}
\tau s^2\int_{W_3} \psi^3 e^{2\tau \psi} \big( \eta^2 v^2+ | D(\eta v)|^2\big )&\geq C \tau s^2 \int_{W_3} e^{2\tau \psi}\big( \eta^2 v^2+ | D\eta v +\eta D v)|^2\big ) \nonumber \\
&\geq \frac{C \tau}{\delta^2} e^{2\tau \psi(\frac{\delta}{4})} \|v\|^2_{H^1(W_1)}.
\label{three-3-1}
\end{align}
Taking (\ref{three-1-1}), (\ref{three-2-1}) and (\ref{three-3-1}) into consideration gives that
\begin{align}
e^{2\tau \psi(\frac{\delta}{4})} \|v\|^2_{H^1(W_1)}\leq \frac{C}{\delta^2} e^{2\tau \psi(\frac{\delta}{2})} \|v\|^2_{H^1(W_3)}+ e^{2\tau}
\|\frac{\partial v}{\partial y_{n+1}}\|^2_{L^2(\mathcal{B}_\delta)}.
\end{align}
Since $\{z_i \}$ is an equidistributed sequence, $W_j(z_i)\cap W_j(z_{i'})=\emptyset$ for $z_{i}\not=z_{i'}$. Thus,
the above estimates hold for all the translated sets $W_j(z_i)$. We sum up over $i\in Q_L$ and divide it by $e^{2\tau \psi(\frac{\delta}{4})}$ to obtain
\begin{align}\label{three-term}
\|v\|^2_{H^1(P_1(L))}\leq  \frac{C}{\delta^2} e^{2\tau \psi_0} \|v\|^2_{H^1(P_3(L))}+
e^{2\tau \psi_1} \|\frac{\partial v}{\partial y_{n+1}}\|^2_{L^2(D_\delta(L))},
\end{align}
where
\begin{align*}
\psi_0= \psi(\frac{\delta}{2})-\psi(\frac{\delta}{4}) \quad \mbox{and} \quad \psi_1= 1-\psi(\frac{\delta}{4}).
\end{align*} 
Recall that $\psi(\hat{r}) = e^{-s\hat{r}}$ and $s = \frac{C_2}{ \delta}$. Then $\psi_0 = e^{-C_2/2} - e^{-C_2/4} < 0$ and $\psi_1 = 1-e^{-C_2/4} > 0$. Moreover, $\psi_1$ and $\psi_2$ are constants depending only on $n$. 

Due to the fact that $\psi_0<0$, we can choose some large constant $\tau$ such that
\begin{align}
\frac{C}{\delta^2} e^{2\tau \psi_0}  \|v\|_{H^1(P_3(L))}^2 \leq \frac{1}{2} \|v\|^2_{H^1(P_1(L))}.
\label{replace}
\end{align}
Substituting (\ref{replace}) into (\ref{three-term}) gives that
\begin{align}
\|v\|_{H^1(P_1(L))}\leq \sqrt{2} e^{\tau \psi_1} \|\frac{\partial v}{\partial y_{n+1}}\|_{L^2(D_\delta(L))}.
\label{partial-goal}
\end{align}
To achieve the estimate (\ref{replace}), $\tau$ needs to satisfy
\begin{align}
\tau\geq \frac{1}{\psi_0} \log \frac{ \delta \|v\|_{H^1(P_1(L))}} {\sqrt{2C}  \|v\|_{H^1(P_3(L))}}.
\end{align}
Combining with the assumption of $\tau$ in the Carleman estimate (\ref{useew}), we can choose the Carleman parameter $\tau$ as the following
\begin{align}
\tau=C_1(1+\|V_1\|^{\frac{1}{2}}_{W^{1,\infty}(\Lambda_{9\sqrt{n}L})} + \norm{V_2}^{\frac23}_{L^\infty(\Lambda_{9\sqrt{n}L})} )+ \frac{1}{\psi_0} \log \frac{ \delta  \|v\|_{H^1(P_1(L))}} {2C  \|v\|_{H^1(P_3(L))}},
\end{align}
since $\cup_{i\in Q_L}  \mathbb B_\delta(z_i,0)\subset \Lambda_{9\sqrt{n}L}$. 
(The side length $9\sqrt{n}L$ is not optimal, but convenient for later use). With such $\tau$, it follows from (\ref{partial-goal}) that
\begin{align}
& \|v\|_{H^1(P_1(L))} \nonumber \\
& \leq C \exp\big( C_1(1+\mathcal{G}(V_1,V_2,9\sqrt{n}L))\psi_1+ \frac{\psi_1}{\psi_0}  \log \frac{ \delta \|v\|_{H^1(P_1(L))}} {C  \|v\|_{H^1(P_3(L))}} \big )\|\frac{\partial v}{\partial y_{n+1}}\|_{L^2(D_\delta(L))} \nonumber \\
&= C \exp\big( C_1(1+\mathcal{G}(V_1,V_2,9\sqrt{n}L))\psi_1 \big) \big(\frac{ \delta  \|v\|_{H^1(P_1(L))}} {C  \|v\|_{H^1(P_3(L))}}\big)^\frac{\psi_1}{\psi_0} \|\frac{\partial v}{\partial y_{n+1}}\|_{L^2(D_\delta(L))}.
\label{long-c}
\end{align}
Let $\alpha=\frac{\psi_1}{\psi_1-\psi_0}$. Then $1-\alpha=\frac{-\psi_0}{\psi_1-\psi_0}$.  It is easy to see that $0<\alpha<1$ depends only on $n$.
Simplifying the inequality (\ref{long-c}) gives that
\begin{align}
\|v\|_{H^1(P_1(L))}&\leq (\frac{C}{\delta})^\alpha \exp\big( C(1+\mathcal{G}(V_1,V_2,9\sqrt{n}L) )\frac{-\psi_1 \psi_0}{\psi_1-\psi_0} \big)\|v\|^\alpha_{H^1(P_3(L))} \|\frac{\partial v}{\partial y_{n+1}}\|^{1-\alpha}_{L^2(D_\delta(L))} \nonumber \\
&\leq \delta^{-\alpha} \exp\big(C(1+\mathcal{G}(V_1,V_2,9\sqrt{n}L) \big) \|v\|^\alpha_{H^1(P_3(L))} \|\frac{\partial v}{\partial y_{n+1}}\|^{1-\alpha}_{L^2(D_\delta(L))},
\end{align}
where $C$ and $\alpha$ depend only on $n$.
\end{proof}

Next, we will use another quantitative local Carleman estimates to obtain a local quantitative three-ball inequality.
Let $r=r(y)$ be the  distance from origin to $y$. 
 We construct the weight function $\hat{\psi}$ as follows.  
Set $$\hat{\psi}(r) = \hat{\psi}(r(y))=-g(\log r(y))$$ with $g(t)=t+\log t^2$. 
We state the following quantitative Carleman estimate without a proof.
\begin{proposition}[Second Carleman estimates]
There exist positive constants $C_1$, $C_0$ and small $r_0$ depending on $n$, such that for any $w\in C^{\infty}_{0}(
\mathbb B_{r_0}(0) \backslash \{0\}), \hat{V} = \hat{V}_1 + \hat{V}_2$,  and
\begin{align} \tau>C_1\big(1+ \|\hat{V}_1\|^\frac{1}{2}_{W^{1, \infty}(\mathbb{B}_{r_0})} + \|\hat{V}_2\|^\frac23_{L^\infty(\mathbb{B}_{r_0} )} \big),
\label{assum-tau}
\end{align}
one has
\begin{align}
& C_0\|r^2 e^{\tau\hat{\psi}} \big(-\triangle w + \hat{V}(x) w\big)\|^2_{L^2(\mathbb{B}_{r_0})} \nonumber \\ 
& \qquad \geq  \tau^3\|e^{\tau\hat{\psi}} {(\log
r)}^{-1}
w \|^2_{L^2(\mathbb{B}_{r_0})}
+\tau\|r e^{\tau\hat{\psi}} {(\log r)}^{-1}D w \|^2_{L^2(\mathbb{B}_{r_0})}.
\label{carl} \end{align} \label{pro1}
\end{proposition}

Similar Carleman estimates with quantitative lower bounds for the parameter $\tau$ have been obtained in, e.g., \cite{DF,Lau, Zh}. Interested readers may refer to these references for the detailed proofs. In particular, the proof of our Proposition \ref{pro1} with $\hat{V}_2 = 0$ can be found in \cite[Proposition 1]{Zh}. If we add a bounded measurable potential function $\hat{V}_2$, the Carleman estimate (\ref{carl}) holds with the lower bound for $\tau>C_1\big(1+ \|\hat{V}_1\|^\frac{1}{2}_{W^{1, \infty}(\mathbb{B}_{r_0})} + \|\hat{V}_2\|^\frac23_{L^\infty(\mathbb{B}_{r_0} )} \big)$ as the same way in the proof of Proposition \ref{Pro4} in the Appendix. The appearance of $\|\hat{V}_2\|^\frac23_{L^\infty(\mathbb{B}_{r_0} )}$ in the lower bound of $\tau$ is well-known; see, e.g., \cite{K}.

The Carleman estimate (\ref{carl}) holds for some small $r_0>0$. By rescaling, we show that this Carleman estimate holds for any ball $\mathbb B_{R}$ with $R>r_0$.

Let $x=\frac{y r_0}{R}$ and $\hat{V}(x)=\frac{R^2 V(y)}{r_0^2}$. Set $u(y)=w(x)$.
Then
\begin{align}
-\triangle w(x)+\hat{V}(x) w(x)=-\frac{R^2}{r_0^2}\triangle u(y)+\frac{R^2}{r_0^2} V(y) u(y)=\frac{R^2}{r_0^2}(-\triangle u(y)+V(y) u(y)).
\end{align}
The Carleman estimate (\ref{carl}) and a change of variables imply
\begin{align}\label{carl-2}
& C_0 \int_{\mathbb B_R} {|y|^4} e^{2\tau\hat{\psi}(\frac{r_0|y|}{R})} \big(-\triangle u + {V}(y) u\big)^2 \\
& \quad \geq  \tau^3 \int_{\mathbb B_R}  e^{2\tau\hat{\psi}(\frac{r_0|y|}{R})} (\log \frac{r_0|y|}{R})^{-2} u^2 +\tau \int_{\mathbb B_R} |y|^2 e^{2\tau\hat{\psi}(\frac{r_0|y|}{R})} | D u|^2  (\log \frac{r_0 |y|}{R})^{-2} \nonumber
\end{align}
for all $\tau$ satisfying
\begin{align}
\tau &>C_1\big(1+ \big( \frac{R}{r_0}\big)^{\frac32} \|V_1\|^\frac{1}{2}_{W^{1, \infty}(\mathbb{B}_{R})} + \big( \frac{R}{r_0}\big)^{\frac43} \|V_2\|^\frac23_{L^\infty(\mathbb{B}_{R} )} \big) \nonumber \\ &\ge C_1\big(1+ \|\hat{V}_1\|^\frac{1}{2}_{W^{1, \infty}(\mathbb{B}_{r_0})} + \|\hat{V}_2\|^\frac23_{L^\infty(\mathbb{B}_{r_0} )} \big).
\end{align}
Since $r_0$ is fixed and $R$ will be fixed later, in order to have the Carleman estimate (\ref{carl-2}), we only need to assume that
\begin{align}
\tau >C\big(1+ \|V_1\|^\frac{1}{2}_{W^{1, \infty}(\mathbb{B}_{R})} + \|V_2\|^\frac23_{L^\infty(\mathbb{B}_{R} )} \big),
\label{assum-w}
\end{align}
where $C = C_1(R/r_0)^{\frac32}$.

Thanks to the quantitative Carleman estimate (\ref{carl-2}), we are able to derive a quantitative three-ball inequality.
 Let $\delta \in (0,\frac12)$ and
 \begin{align*} R_1=\frac{1}{16}\delta  \ \mbox{and} \ r_1=\frac{1}{32}\delta,\ \ R_2=3\sqrt{n} \ \mbox{and} \ r_2=\frac{1}{2}, \ \ R_3=9 \sqrt{n} \ \mbox{and}  \ r_3=6 \sqrt{n}.
 \end{align*}
  We introduce annular regions
  \begin{align*} 
  S_1=\mathbb B_{R_1}\backslash \overline{\mathbb B_{r_1}}, \quad  S_2=\mathbb B_{R_2}\backslash \overline{\mathbb B_{r_2}}, \quad S_3=\mathbb B_{R_3}\backslash \overline{\mathbb B_{r_3}}.
  \end{align*} 
We choose $R = 2R_3 = 18\sqrt{n}$. Define
\begin{equation}
    X_1 = \Lambda_L \times [-1,1] \quad \text{and} \quad \tilde{X}_{R_3} = \Lambda_{L+R_3} \times [-R_3, R_3].
\end{equation}

\begin{lemma}
Let $\delta \in (0,\frac12)$. Let $v$ be the solution of (\ref{eqn-ste}) which is odd with respect to $y_{n+1}$. There exist $C>0$ depending only on $n$, $0<\alpha_1 <1$ depending on $\delta$ and $n$ (see \eqref{def.alpha1}) such that
\begin{align}
\| v\|_{H^1(X_1)} \leq \delta^{-2\alpha_1} \exp\big(C(1+\mathcal{G}(V_1, V_2, 9\sqrt{n}L)) \big) \| v\|^{1-\alpha_1} _{H^1(\tilde{X}_{R_3})} \| v\|^{\alpha_1}_{H^1(P_1(L))},
\label{threeball-2}
\end{align}
where $\mathcal{G}(V_1,V_2,L)$ is given by \eqref{defG}.
\label{lemma-3}
\end{lemma}

\begin{proof}
We choose a smooth cut-off function $0<\eta<1$ as follows:
\begin{itemize}
\item $\eta(r)=0$ \ \ \mbox{if} \ $r(y)<r_1$ \ \mbox{or} \  $r(y)>R_3$, \medskip
\item $\eta(r)=1$ \ \ \mbox{if} \ $R_1<r(y)<r_3$, \medskip
\item $\| D \eta\|_{L^\infty(S_1)}\leq \frac{C}{\delta}$ and $\|\triangle \eta\|_{L^\infty(S_1)}\leq \frac{C}{\delta^2}$, \ \medskip
\item $\|D \eta\|_{L^\infty(S_3)}\leq C$ and $\|\triangle \eta\|_{L^\infty(S_3)}\leq {C}$ \ \  \mbox{if} \ $r_3<r(y)<R_3$.
\end{itemize}
From the equation (\ref{eqn-ste}), we have
\begin{align}
-\triangle (\eta v)+V \eta v =-\triangle \eta v+ 2D \eta \cdot D v.
\end{align}
 We substitute $u=\eta v $ into the Carleman estimate (\ref{carl-2}) to have
 \begin{align}\label{full-in}
& C_0 \int_{\mathbb B_{R_3}} {|y|^4} e^{2\tau\hat{\psi}(\frac{r_0|y|}{R})} (-\triangle \eta v+ 2D \eta \cdot D v)^2
\\ \quad &\geq  \tau^3 \int_{\mathbb B_{R_3}}  e^{2\tau\hat{\psi}(\frac{r_0|y|}{R})} (\log \frac{r_0|y|}{R})^{-2} (\eta v)^2 \nonumber +\tau \int_{\mathbb B_{R_3}} |y|^2 e^{2\tau\hat{\psi}(\frac{r_0|y|}{R})} | D (\eta v)|^2  (\log \frac{r_0|y|}{R})^{-2}.
\end{align}

We first study the left hand side of (\ref{full-in}). The property of $\eta$ implies that
\begin{align}
 & \int_{\mathbb B_{R_3}} {|y|^4} e^{2\tau\hat{\psi}(\frac{r_0 |y|}{R})} (-\triangle \eta v+ 2D \eta \cdot D v)^2 \nonumber\\ & \leq
 CR_3^4 \int_{\mathbb B_{R_3}}  e^{2\tau\hat{\psi}(\frac{r_0 |y|}{R})} |\triangle \eta v |^2+  CR_3^4 \int_{\mathbb B_{R_3}}  e^{2\tau\hat{\psi}(\frac{r_0 |y|}{R})} |D\eta|^2|Dv|^2\nonumber \\
  &=A_1+A_2.
  \label{AAA}
\end{align}
Applying the properties of $\hat{\psi}$ and $\eta$ yields that
 \begin{align}
 A_1\leq  \frac{CR_3^4}{\delta^4} e^{2\tau\hat{\psi}(\frac{r_0 r_1}{R})} \int_{S_1} |v |^2+  CR_3^4  e^{2\tau\hat{\psi}(\frac{r_0 r_3}{R})} \int_{S_3}  |v |^2.
 \label{AA1}
 \end{align}
 Similarly, it holds that
  \begin{align}
 A_2\leq  \frac{CR_3^4}{\delta^2} e^{2\tau\hat{\psi}(\frac{r_0 r_1}{R})} \int_{S_1} |Dv |^2+  CR_3^4  e^{2\tau\hat{\psi}(\frac{r_0 r_3}{R})} \int_{S_3}  |Dv |^2.
 \label{AA2}
 \end{align}
 For the right-hand side of (\ref{full-in}), we have
 \begin{align}
 & \tau^3 \int_{\mathbb B_{R_3}}  e^{2\tau\hat{\psi}(\frac{r_0|y|}{R})} (\log \frac{r_0|y|}{R})^{-2} (\eta v)^2
+\tau \int_{\mathbb B_{R_3}} |y|^2 e^{2\tau\hat{\psi}(\frac{r_0|y|}{R})} | D (\eta v)|^2  (\log \frac{r_0|y|}{R})^{-2} \nonumber \\
&\geq \tau (\log \frac{r_0 r_2}{R})^{-2} e^{2\tau\hat{\psi}(\frac{r_0 R_2}{R})}  \big(\int_{S_2}   v^2 +  r_2^2 \int_{S_2}  |D v|^2\big)\nonumber \\
 & \geq r_2^2(\log \frac{r_0 r_2}{R})^{-2}e^{2\tau\hat{\psi}(\frac{r_0 R_2}{R})}  \| v\|^2_{H^1(S_2)}.
 \label{right-low}
 \end{align}
Combining these estimates (\ref{full-in}) -- (\ref{right-low}) together gives
\begin{align}
 r_2^2(\log \frac{r_0 r_2}{R})^{-2}e^{2\tau\hat{\psi}(\frac{r_0 R_2}{R})}  \| v\|^2_{H^1(S_2)}\leq \frac{CR_3^4}{\delta^4} e^{2\tau\hat{\psi}(\frac{r_0 r_1}{R})}  \| v\|^2_{H^1(S_1)}+ CR_3^4  e^{2\tau\hat{\psi}(\frac{r_0 r_3}{R})} \| v\|^2_{H^1(S_3)}.
\end{align}
Recall that $r_0, r_2,r_3, R_2, R_3$ and $R$ are all fixed constants and $r_1 = \frac{\delta}{32}$. It follows from the last inequality that
\begin{align}
 \| v\|^2_{H^1(S_2)}\leq \frac{C} {\delta^4} e^{2\tau\hat{\psi}(\frac{r_0 r_1}{R})-2\tau\hat{\psi}(\frac{r_0 R_2}{R})}\| v\|^2_{H^1(S_1)}+
  C e^{2\tau\hat{\psi}(\frac{r_0 r_3}{R})-2\tau\hat{\psi}(\frac{r_0 R_2}{R})} \| v\|^2_{H^1(S_3)}.
\end{align}
Note that this estimate is translation invariant. Then we can apply it to all translated $S_j(z_i)$ with $(1, \delta)$-equidistributed sequence ${z_i}$ and $i\in Q_L$. Then
\begin{align}
\sum_{i\in Q_L} \| v\|^2_{H^1(S_2(z_i))}& \leq \frac{C} {\delta^4} e^{2\tau\hat{\psi}(\frac{r_0 r_1}{R})-2\tau\hat{\psi}(\frac{r_0 R_2}{R})}\sum_{i\in Q_L} \| v\|^2_{H^1(S_1(z_i))} \nonumber \\& \qquad+
  C  e^{2\tau\hat{\psi}(\frac{r_0 r_3}{R})-2\tau\hat{\psi}(\frac{r_0 R_2}{R})} \sum_{i\in Q_L} \| v\|^2_{H^1(S_3(z_i))}.
  \label{three-1}
\end{align}
We can check that 
\begin{align} 
X_1 = \Lambda_L\times [-1, 1] \subset \bigcup_{i\in Q_L} S_2(z_i).
\label{state-2}
\end{align} 
The proof of the statement (\ref{state-2}) has been shown in \cite[Lemma 3.3]{NTTV}. For the completeness of presentation, we include its proof in Lemma \ref{lemma-app} in the Appendix. Thus,
\begin{align}
\| v\|^2_{H^1(X_1)} \leq \sum_{i\in Q_L} \| v\|^2_{H^1(S_2(z_i))}.
\label{three-2}
\end{align}
We can see that $W_1(z_i)\cap W_1(z_{i'})=\emptyset$ for $i\not=i'$.
Since $S_1(z_i)\cap \mathbb R^{n+1}_+\subset W_1(z_i)$ and $v$ is odd with respect to $y_{n+1}$ variable, we get
\begin{align}
\sum_{i\in Q_L} \| v\|^2_{H^1(S_1(z_i))} \leq 2 \sum_{j\in Q_L} \| v\|^2_{H^1(W_1(z_i))} \leq 2\| v\|^2_{H^1(P_1(L))}.
\label{three-3}
\end{align}
From the geometric overlapping of the sets $S_3(z_i)$, it holds that
\begin{align}
 \sum_{i\in Q_L} \| v\|^2_{H^1(S_3(z_i))}\leq \gamma_n \| v\|^2_{H^1(\cup_{i\in Q_L} S_3(z_i))},
 \label{three-4}
\end{align}
where $\gamma_n$ depends only on $n$. The statement of (\ref{three-4}) is verified in Lemma \ref{lemma-app} in the Appendix.
Since $\tilde{X}_{R_3}= \Lambda_{L+R_3}\times [-R_3, \ R_3] \subset \cup_{i\in Q_L} S_3(z_i)$, we have
\begin{align}
\sum_{i\in Q_L} \| v\|^2_{H^1(S_3(z_j))}\leq \gamma_n \| v\|^2_{H^1(\tilde{X}_{R_3})}.
\label{three-5}
\end{align}
Taking the estimates (\ref{three-1}) -- (\ref{three-5}) into consideration, we obtain
\begin{align}\label{est.vH1inX1}
\| v\|_{H^1(X_1)} \leq \frac{C} {\delta^2} e^{\tau(\hat{\psi}(\frac{r_0r_1}{R})-\hat{\psi}(\frac{r_0 R_2}{R}))} \| v\|_{H^1(P_1(L))}+
C  e^{\tau(\hat{\psi}(\frac{r_0 r_3}{R})-\hat{\psi}(\frac{r_0 R_2}{R}))} \| v\|_{H^1(\tilde{X}_{R_3})}.
\end{align}
Let $\hat{\kappa}_1=\hat{\psi}(\frac{r_0 r_1}{R})-\hat{\psi}(\frac{r_0 R_2}{R})$ and $\kappa_2=\hat{\psi}(\frac{r_0 r_3}{R})-\hat{\psi}(\frac{r_0 R_2}{R})$. By the definition of $\hat{\psi}$, we have
\begin{equation}
\begin{aligned}
    \hat{\psi}(\frac{r_0 r_1}{R})-\hat{\psi}(\frac{r_0 R_2}{R}) & = \log \frac{R_2}{r_1} + \log \Big( \log(\frac{r_0 R_2}{R}) \Big)^2 -  \log \Big( \log(\frac{r_0 r_1}{R}) \Big)^2  \\
    & \le \log \frac{R_2}{r_1} = \kappa_1:= \log \frac{96 \sqrt{n}}{\delta},
\end{aligned}
\end{equation}
and
\begin{equation}
    \kappa_2 := \hat{\psi}(\frac{r_0 r_3}{R})-\hat{\psi}(\frac{r_0 R_2}{R}) = \hat{\psi}(\frac{r_0}{3})-\hat{\psi}(\frac{r_0}{6}).
\end{equation}
Clearly $\kappa_1>0$ depends on $\delta$ and $n$. Without loss of generality, assume $r_0<e^{-1}$ (since $r_0$ is supposed to be small). Then it is not hard to see $\kappa_2<0$ depends only on $n$.
Using $\kappa_1$ and $\kappa_2$, we obtain from \eqref{est.vH1inX1} that
\begin{align}
\| v\|_{H^1(X_1)} \leq \frac{C} {\delta^2} e^{\tau \kappa_1} \| v\|_{H^1(P_1(L))}+
C  e^{\tau \kappa_2} \| v\|_{H^1(\tilde{X}_{R_3})}.
\label{three-three}
\end{align}
We choose $\tau$ such that
\begin{align*}
C  e^{\tau \kappa_2} \| v\|_{H^1(\tilde{X}_{R_3})}\leq \frac{1}{2} \| v\|_{H^1(X_1)}.
\label{tau-1}
\end{align*}
That is
\begin{align}
\tau\geq \frac{1}{\kappa_2} \log \frac{\| v\|_{H^1(X_1)}} {2C\| v\|_{H^1(\tilde{X}_{R_3})} } .
\end{align}
Therefore,  (\ref{three-three}) turns into
\begin{align}
\| v\|_{H^1(X_1)} \leq \frac{2C} {\delta^2} e^{\tau \kappa_1} \| v\|_{H^1(P_1(L))}.
\label{plug-1}
\end{align}
To achieve the estimate (\ref{tau-1}) and make use of the assumption (\ref{assum-w}), we select $\tau$ as
\begin{align}
\tau=\frac{1}{\kappa_2} \log \frac{ \| v\|_{H^1(X_1)}}{2C \| v\|_{H^1(\tilde{X}_{R_3})} }+ C(1+\mathcal{G}(V_1, V_2, 9\sqrt{n}L) ),
\end{align}
due to $\tilde{X}_{R_3}\subset \Lambda_{9\sqrt{n}L}$.
We substitute the above $\tau$ into (\ref{plug-1}) to get
\begin{align}
\| v\|_{H^1(X_1)} \leq \frac{2C} {\delta^2}  \exp\big(\frac{\kappa_1}{\kappa_2}\log \frac{ \| v\|_{H^1(X_1)}}{2C \| v\|_{H^1(\tilde{X}_{R_3})} } +C\kappa_1(1+\mathcal{G}(V_1, V_2, 9\sqrt{n}L) ) \big) \| v\|_{H^1(P_1(L))}.
\end{align}
We simplify the last estimate to get
\begin{align}
\| v\|_{H^1(X_1)} \leq (\frac{2C} {\delta^2})^{\frac{-\kappa_2}{\kappa_1-\kappa_2}}   \exp\big(C\frac{-\kappa_1\kappa_2}{\kappa_1-\kappa_2}(1+\mathcal{G}(V_1, V_2, 9\sqrt{n}L) ) \big) \| v\|^{\frac{\kappa_1}{\kappa_1-\kappa_2}} _{H^1(\tilde{X}_{R_3})} \| v\|^{\frac{-\kappa_2}{\kappa_1-\kappa_2}}_{H^1(P_1(L))}.
\end{align}
Let ${\frac{-\kappa_2}{\kappa_1-\kappa_2}} =\alpha_1$. Then $\frac{\kappa_1}{\kappa_1-\kappa_2}=1-\alpha_1$. Obviously,
\begin{equation}\label{def.alpha1}
    0<\alpha_1 = \frac{|\kappa_2|}{|\log \delta| + \log(96\sqrt{n}) + |\kappa_2| }<1.
\end{equation}
Hence, we obtain
\begin{align}
\| v\|_{H^1(X_1)} &\leq (\frac{2C} {\delta^2})^{\alpha_1}   \exp\big(-C\kappa_2(1-\alpha_1)(1+\mathcal{G}(V_1, V_2, 9\sqrt{n}L) ) \big) \| v\|^{1-\alpha_1} _{H^1(\tilde{X}_{R_3})} \| v\|^{\alpha_1}_{H^1(P_1(L))} \nonumber \\
& \leq \delta^{-2\alpha_1}  \exp\big(C(1+\mathcal{G}(V_1, V_2, 9\sqrt{n}L) ) \big)  \| v\|^{1-\alpha_1} _{H^1(\tilde{X}_{R_3})} \| v\|^{\alpha_1}_{H^1(P_1(L))}.
\end{align}
This completes the proof of the lemma.
 \end{proof}

\section{Proof of the spectral inequality }
In this section, we will show the proof of the spectral inequality (\ref{aim-res}) in Theorem \ref{th1}.
We make use of the lifting argument (or so-called ghost dimension construction) to get rid of  $\lambda_k$. Since $V(x)$ is nonnegative and growing to infinity, all the eigenvalues $\lambda_k$ are positive. Let $\phi \in Ran(P_\lambda(H))$ be given by (\ref{phi-s}) with eigenpairs $(\phi_k,\lambda_k)$ satisfying (\ref{eigen-k}). 
Define
\begin{align}
\Phi(x, x_{n+1})=\sum_{0<\lambda_k\leq \lambda} \alpha_k \phi_k(x) \frac{\sinh(\sqrt{\lambda_k} x_{n+1}) }{\sqrt{\lambda_k}}.
\label{phi-c}
\end{align}
Then $\Phi(x, x_{n+1})$ satisfies the equation
\begin{align}
-{\triangle} \Phi+V(x) \Phi=0 \quad \mbox{in} \ \mathbb R^{n+1},
\label{PPhi}
\end{align}
$D_{n+1} \Phi(x, 0)=\phi(x)$, and $\Phi(x, 0)=0$, where $\triangle \Phi =\sum^{n+1}_{i=1} D^2_i \Phi$.

The following estimate for $\Phi$ is standard. We include its proof here for the completeness of presentation. Interested readers may also refer to, e.g., \cite{JL} or \cite{NTTV} for a detailed account.
\begin{lemma}
Let $\phi\in Ran (P_\lambda(H))$ and $\Phi$ be given in (\ref{phi-c}). For any $\rho>0$, we have
\begin{align}
2\rho \|\phi\|^2_{L^2(\mathbb R^n)}\leq \|\Phi\|^2_{H^1(\mathbb R^{n}\times (-\rho, \rho))}\leq 2\rho(1+\frac{\rho^2}{3}(1+\lambda) ) e^{2\rho \sqrt{\lambda}} \|\phi\|^2_{L^2(\mathbb R^n)}.
\label{com-phi}
\end{align}
\label{lemma-1}
\end{lemma}
\begin{proof} 
Using the basic properties of $\cosh$ and $\sinh$, we have
\begin{align}
|\sinh (\sqrt{\lambda_k} x_{n+1})|\leq |\sqrt{\lambda_k} x_{n+1} \cosh(\sqrt{\lambda_k}  x_{n+1})|,
\label{calcu-1}
\end{align}
and
\begin{align}
1\leq \cosh( \sqrt{\lambda_k}  x_{n+1})\leq e^{\sqrt{\lambda_k}  |x_{n+1}|}.
\label{calcu-2}
\end{align}
From (\ref{calcu-1}), (\ref{calcu-2}) and the orthogonality of $\phi_k$, we obtain
\begin{align}
\|\Phi(\cdot, x_{n+1})\|^2_{L^2(\mathbb R^n)}&= \sum_{0<\lambda_k\le \lambda} \frac{\sinh^2 (\sqrt{\lambda_k} x_{n+1})}{\lambda_k} |\alpha_k|^2 \nonumber \\
&\le \sum_{0<\lambda_k\le \lambda} e^{2\sqrt{\lambda_k} |x_{n+1}|} x^2_{n+1} |\alpha_k|^2 \nonumber \\
 &=  e^{2\sqrt{\lambda} |x_{n+1}|} x^2_{n+1} \|\phi\|^2_{L^2(\mathbb R^n)},
 \label{eval-P}
\end{align}
where $\phi$ is given by (\ref{phi-s}).
We integrate the last inequality with respect to $x_{n+1}$ in $(-\rho, \rho)$ to get
\begin{align}
\|\Phi(\cdot, x_{n+1})\|^2_{L^2(\mathbb R^n\times (-\rho, \rho))}\leq \frac{2}{3}\rho^3 e^{2\sqrt{\lambda}\rho }\|\phi\|^2_{L^2(\mathbb R^n)}.
\label{idend-1}
\end{align}
From the definition of $\Phi$, (\ref{calcu-2}) and the orthogonality of $\phi_k$, we have
\begin{align}
\|D_{n+1}\Phi\|^2_{L^2(\mathbb R^n)}&=\|\sum_{0<\lambda_k\le \lambda} \alpha_k\phi_k \cosh(\sqrt{\lambda_k} x_{n+1}) \|^2_{L^2(\mathbb R^n)} \nonumber \\
&\leq e^{2\sqrt{\lambda} |x_{n+1}|} \|\phi\|_{L^2(\mathbb R^n)}^2.
\end{align}
Integrating in $x_{n+1}$ over $(-\rho, \rho)$ gives
\begin{align}
\|D_{n+1}\Phi\|^2_{L^2(\mathbb R^n\times (-\rho, \rho))}\leq 2 \rho e^{2\sqrt{\lambda}\rho }\|\phi\|_{L^2(\mathbb R^n)}^2.
\label{idend-2}
\end{align}

Next, to estimate $D_i \Phi$ for $1\le i\le n$, we consider the following equation
\begin{align}
-\triangle \Phi(\cdot, x_{n+1})+V(x)\Phi(\cdot, x_{n+1})=\sum_{0<\lambda_k\leq \lambda} \alpha_k \sqrt{\lambda_k} \phi_k(\cdot) \sinh( \sqrt{\lambda_k} x_{n+1}), 
\label{eqn-sti}
\end{align}
where $\triangle$ is the $n$-dimensional Laplace operator in $x$ (without $x_{n+1}$).
Integrating (\ref{eqn-sti}) against $\Phi(\cdot,x_{n+1})$ over $\R^n$ and using the orthogonality of $\phi_k$, we get
\begin{align}
\sum^n_{i=1} \| D_i \Phi\|^2_{L^2(\mathbb R^n)}&\leq \int_{\mathbb R^n}\sum^n_{i=1} | D_i \Phi|^2 +V(x) \Phi^2 \nonumber \\
& \leq \sum_{0<\lambda_k\le \lambda} |\alpha_k|^2 \sinh^2(\sqrt{\lambda_k} x_{n+1}) \nonumber \\
&\leq \lambda x_{n+1}^2 e^{2\sqrt{ \lambda} |x_{n+1}| } \|\phi\|_{L^2(\mathbb R^n)}^2.
\label{idend-3}
\end{align}
Thus, by integrating in $x_{n+1}$ over $(-\rho,\rho)$, we obtain
\begin{align}
\sum^n_{i=1} \| D_i \Phi\|^2_{L^2(\mathbb R^n \times (-\rho, \rho) )} \leq \frac{2}{3}\lambda\rho^3 e^{2\sqrt{\lambda}\rho }\|\phi\|_{L^2(\mathbb R^n)}^2.
\label{idend-33}
\end{align}
A combination of (\ref{idend-1}), (\ref{idend-2}) and (\ref{idend-33}) gives the second inequality in (\ref{com-phi}).

From (\ref{calcu-2}) and the orthogonality of $\phi_k$, it holds that
\begin{align}
\|D_{n+1}\Phi(\cdot,x_{n+1})\|^2_{L^2(\mathbb R^n)}&=\|\sum_{0<\lambda_k\le \lambda} \alpha_k\phi_k \cosh(\sqrt{\lambda_k} x_{n+1}) \|^2_{L^2(\mathbb R^n)} \nonumber \\
&\geq \|\phi\|_{L^2(\mathbb R^n)}^2.
\end{align}
Note that this estimate is uniform in $x_{n+1}$. Thus, integrating in $x_{n+1}$ over $(-\rho,\rho)$ leads to the first inequality in (\ref{com-phi}).
\end{proof}

The following proposition quantifies the decay property of the linear combination of eigenfunctions $\phi$, which has been studied in, e.g., \cite{A,BS,GY} and \cite{DSV}.
\begin{proposition}
There exists a constant $\hat{C}$, depending on $\beta_1$, $c_1$, $\beta_2$ and $c_2$ such that for all $\lambda \ge \lambda_1$ and $\phi\in Ran(P_\lambda(H))$, we have
\begin{align}
\|\phi\|^2_{H^1(\mathbb R^n\backslash \mathcal B_{\hat{C}\lambda^{1/\beta_1}})} \leq \frac{1}{2} \|\phi\|^2_{L^2(\mathbb R^n)}.
\label{decay-cr}
\end{align}
\label{propo-new}
\end{proposition}

\begin{remark}\label{rmk.Decay}
A proof of \eqref{decay-cr} can be found in \cite[Theorem 1.4]{DSV} or \cite[Lemma 2.1]{GY}. However, in their proofs a mild condition for $DV(x)$ is needed because they took differentiation to the entire equation of eigenfunctions in order to get the estimate for $D \phi_k$. We here claim that 
    this decay property depends only on the growth bound of $V$. In particular, the differentiability of $V$ is not relevant. The argument is a simple modification of the proof of \cite[Theorem 1.4]{DSV} or \cite[Lemma 2.1]{GY}. By checking their proofs, we see that if $(-\triangle + V) \phi_k = \lambda_k \phi_k$, then for $R$ satisfying $R^{\beta_1}>\max\{ (\lambda_k+2)/c_1, 1\}$,
    \begin{equation}\label{phik.decay}
        \| e^{|x|/2}\phi_k \|_{L^2(\R^n)} \le 7e^{R+1} \| \phi_k \|_{L^2(\R^n)}.
    \end{equation}
    This estimate does not require differentiability of $V$. Now, to get the estimate for $D\phi_k$, we only need to apply local Caccioppoli inequalities. Indeed, for any $\mathcal{B}_{1}(z) \subset \R^n$, the Caccioppoli inequality implies
    \begin{equation}
    \begin{aligned}
        \| D \phi_k \|_{L^2(\mathcal{B}_{1}(z))}^2 & \le C (1 + \lambda_k) \| \phi_k \|_{L^2(\mathcal{B}_{2}(z))}^2 + \| |V|^{\frac12} \phi_k \|_{L^2(\mathcal{B}_{2}(z))}^2\\
        & \le C\big( 1 + c_1 R^{\beta_2} \big) \| \phi_k \|_{L^2(\mathcal{B}_{2}(z))}^2 + C\| (1+|\cdot|)^{\beta_2/2}\phi_k \|_{L^2(\mathcal{B}_{2}(z))}^2.
    \end{aligned}
    \end{equation}
    Multiplying the inequality by $e^{|z|}$ and using the fact $e^{|z|} \approx e^{|x|}$ for any $x\in \mathcal{B}_{2}(z)$, we have
    \begin{align}
        \| e^{|x|/4} D \phi_k \|_{L^2(\mathcal{B}_{1}(z))}^2\le  &C\big( 1 + c_1 R^{\beta_2} \big) \| e^{|x|/4}\phi_k \|_{L^2(\mathcal{B}_{2}(z))}^2 \nonumber \\ &+ C\| (1+|\cdot|)^{\beta_2/2} e^{|x|/4} \phi_k \|_{L^2(\mathcal{B}_{2}(z))}^2.
    \end{align}
    Covering $\R^n$ by $\mathcal{B}_1(z)$ with finite overlaps, summing over $z$ and using \eqref{phik.decay}, we have
    \begin{equation}\label{Dphik.decay}
    \begin{aligned}
        \| e^{|x|/4} D \phi_k \|_{L^2(\R^n)}^2 & \le C\big( 1 + c_1 R^{\beta_2} \big) \| e^{|x|/4}\phi_k \|_{L^2(\R^n)}^2 + C\| (1+|\cdot|)^{\beta_2/2} e^{|x|/4} \phi_k \|_{L^2(\R^n)}^2 \\
        & \le C(1+c_1 R^{\beta_2}) e^{2(R+1)} \norm{\phi_k}_{L^2(\R^n)}^2.
    \end{aligned}
    \end{equation}
    Once the exponential decay estimates for both $\phi_k$ and $D\phi_k$ are established, the proof of \eqref{decay-cr} follows from the same argument as, say, \cite[Theorem 1.4]{DSV}.
\end{remark}

Based on (\ref{decay-cr}), we can show that the global $H^1$ norm of $\Phi$ can be controlled by its local norm.
In fact, it follows from (\ref{decay-cr}) that
\begin{align}
\|\phi\|^2_{H^1(\mathbb R^n\backslash \mathcal B_{\hat{C}\lambda^{1/\beta_1}})}\leq  \|\phi\|^2_{L^2(\mathcal B_{\hat{C}\lambda^{1/\beta_1}})}.
\label{cut-o}
\end{align}
This yields
\begin{align}
\|\phi\|^2_{H^1(\mathbb R^n)}\leq 2 \|\phi\|^2_{H^1(\mathcal B_{\hat{C}\lambda^{1/\beta_1}})},
\label{cut-off}
\end{align}
and
\begin{align}
\|\phi\|^2_{L^2 (\mathbb R^n)}\leq 2 \|\phi\|^2_{L^2(\mathcal B_{\hat{C}\lambda^{1/\beta_1 }})}.
\label{L2}
\end{align}

Since $\Phi(\cdot, x_{n+1})\in Ran(P_\lambda(H))$, the estimate (\ref{cut-off}) implies
\begin{align}
\|\Phi\|^2_{H^1(\mathbb R^n)}\leq 2 \|\Phi\|^2_{H^1(\mathcal B_{\hat{C}\lambda^{1/\beta_1}})}.
\label{inte-x}
\end{align}
Since $D_{n+1} \Phi(\cdot, x_{n+1})\in Ran(P_\lambda(H))$, the estimate (\ref{L2}) gives that
\begin{align}\label{inte-xn+1}
\|D_{n+1} \Phi\|^2_{L^2 (\mathbb R^n)}\leq 2 \|D_{n+1} \Phi\|^2_{L^2(\mathcal B_{\hat{C}\lambda^{1/\beta_1}})}.
\end{align}
Combining both \eqref{inte-x} and \eqref{inte-xn+1} and integrating in $x_{n+1}$ over $(-1, 1)$, we obtain
\begin{align}
\|\Phi\|_{H^1 (\mathbb R^n\times (-1, 1))}^2 \leq 2 \| \Phi\|_{H^1(\mathcal B_{\hat{C}\lambda^{1/\beta_1}}\times (-1, 1))}^2.
\label{cut-est}
\end{align}

We are ready to provide the proof of Theorem \ref{th1}.

\begin{proof}[Proof of Theorem \ref{th1}]
Let $L = 2\lceil \hat{C} \lambda^{1/\beta_1} \rceil + 1$, where $\hat{C}$ is given in Proposition \ref{propo-new} and $\lceil a \rceil$ means the smallest integer that does not exceed $a$. Recall that $\Lambda_L = [-L/2, L/2]^n$. Then it is obvious that $\mathbb{B}_{\hat{C} \lambda^{1/\beta_1}} \subset \Lambda_L$. Moreover, we can decompose $\Lambda_L$ as
\begin{equation}
    \Lambda_L = \bigcup_{k\in \Lambda_L \cap \Z^n} \Big(k + \big[-\frac12,\frac12 \big]^n \Big).
\end{equation}
Observe that $|k| \le \sqrt{n} \lceil \hat{C} \lambda^{1/\beta_1} \rceil$ for each $k\in \Lambda_L \cap \Z^n$.
Let $\gamma \in (0,\frac12)$ be as in \eqref{geom-om} and 
\begin{equation}\label{def.delta}
    \delta := \gamma^{1+ (\sqrt{n} \lceil \hat{C} \lambda^{1/\beta_1} \rceil)^\sigma } \le \gamma^{1+|k|^\sigma }, \quad \text{for all }  k\in \Lambda_L \cap \Z^n.
\end{equation}
The assumption (\ref{geom-om}) on $\Omega$ shows that the intersection $\Omega\cap (k+[-\frac{1}{2}, \frac{1}{2}]^n)$ contains a ball centered at some $z_k \in k+[-\frac{1}{2}, \frac{1}{2}]^n$ with radius $\delta $ for each $k$. Thus, the sequence $\{z_k\}$ in the set $\Omega\cap\Lambda_L$ is $(1, \delta)$-equidistributed in $\Lambda_L$.

Now we show an interpolation inequality, using the previous two three-ball inequalities. We replace $v$ in the equation (\ref{eqn-ste}) by $\Phi$ in the equation (\ref{PPhi}). Note that $\Phi$ is odd in $x_{n+1}$. We combine (\ref{threeball-1}) in Lemma \ref{lemma-2} and (\ref{threeball-2}) in Lemma \ref{lemma-3} with $\delta$ and $L$ defined above to get
\begin{align*}
\| \Phi\|_{H^1(X_1)} &\leq \delta^{-2\alpha_1} \exp\big(C(1+ \mathcal{G}(V_1,V_2,9\sqrt{n} L ) ) \big) \| \Phi\|^{1-\alpha_1} _{H^1(\tilde{X}_{R_3})} \|\Phi\|^{\alpha_1}_{H^1(P_1(L))} \nonumber \\
&\leq \delta^{-2\alpha_1-\alpha \alpha_1} \exp\big(C(1+ \mathcal{G}(V_1,V_2,9\sqrt{n} L ) ) \big) \|\Phi\|^{\alpha \alpha_1}_{H^1(P_3(L))} \|\frac{\partial \Phi}{\partial y_{n+1}}\|^{\alpha_1 (1-\alpha)}_{L^2(D_\delta(L))}\|\Phi\|^{1-\alpha_1}_{H^1(\tilde{X}_{R_3})}\nonumber \\
&\leq \delta^{-3\alpha_1} \exp\big(C(1+ \mathcal{G}(V_1,V_2,9\sqrt{n} L ) ) \big) \|\phi\|^{\hat{\alpha}}_{L^2(D_\delta(L))} \|\Phi\|^{1-\hat{\alpha}}_{H^1(\tilde{X}_{R_3})},
\end{align*}
where $\hat{\alpha}=\alpha_1(1-\alpha)$ and we have used the facts $P_3(L)\subset \tilde{X}_{R_3}$ and $\frac{\partial \Phi}{\partial y_{n+1}}(\cdot,0)=\phi$. Recall the value of $\alpha_1$ in Lemma \ref{lemma-3} is defined by \eqref{def.alpha1}. Thus $\alpha_1 \approx \hat{\alpha} \approx \frac{1}{|\log \delta|}$ for any $\delta\in (0,\frac12)$. Hence $\delta^{-3\alpha_1} \le C$ and then
\begin{align}\label{est.PhiOnX1}
\| \Phi\|_{H^1(X_1)} \leq \exp\big(C(1+ \mathcal{G}(V_1,V_2,9\sqrt{n} L ) ) \big) \|\phi\|^{\hat{\alpha}}_{L^2(\Omega\cap \Lambda_L)} \|\Phi\|^{1-\hat{\alpha}}_{H^1(\tilde{X}_{R_3})},
\end{align}
where we have also used the fact $D_\delta(L)\subset \Omega\cap \Lambda_L$.

Recalling the definitions of $L = 2\lceil \hat{C} \lambda^{1/\beta_1} \rceil + 1$ and $R_3 = 9\sqrt{n}$, we have $\Lambda_L\subset \Lambda_{L+R_3}\subset \Lambda_{9\sqrt{n}L}$. By Assumption (A), we have
\begin{equation}
    \mathcal{G}(V_1,V_2,9\sqrt{n} L ) \le C(1+\frac{9}{2} nL)^{\frac{\beta_2}{2}} \le C_* \lambda^{\frac{\beta_2}{2\beta_1}},
\end{equation}
for $\lambda\ge \lambda_1$.
Therefore, we can write (\ref{est.PhiOnX1}) as
\begin{align}
\| \Phi\|_{H^1(X_1)} \leq   \exp(C_* \lambda^{\frac{\beta_2}{2\beta_1} }) \|\phi\|^{\hat{\alpha}}_{L^2(\Omega\cap \Lambda_L)} \|\Phi\|^{1-\hat{\alpha}}_{H^1(\tilde{X}_{R_3})}.
\label{ready-three}
\end{align}

Next, applying $\rho = R_3$ and $\rho_1$ in Lemma \ref{lemma-1} for upper and lower bounds, respectively, we obtain
\begin{align}
\frac{\|\Phi\|^2_{H^1(\mathbb R^{n}\times (-R_3, R_3))}} {\|\Phi\|^2_{H^1(\mathbb R^{n}\times (-1, 1))}  } \leq 2 R_3 (1+\frac{R_3^2}{3}(1+\lambda) ) \exp({2R_3 \sqrt{\lambda}}) \leq  \exp({C_2 \sqrt{\lambda}}),
\end{align}
where $C_2$ depends only on $n$. With the aid of (\ref{cut-est}) and the fact $\mathcal B_{\hat{C}\lambda^{1/\beta_1}}\subset \Lambda_L$, we get
\begin{align}
\|\Phi\|_{H^1(\mathbb R^{n}\times (-R_3, R_3))}&\leq  \exp(\frac12 C_2\sqrt{\lambda} ) \|\Phi\|_{H^1(\mathbb R^{n}\times (-1, 1))} \nonumber \\
&\leq 2 \exp(\frac12 C_2\sqrt{\lambda} )\|\Phi\|_{H^1(\Lambda_L\times (-1, 1))}
\end{align} 
Recall that $X_1 = \Lambda_L\times (-1,1)$. Thanks to the interpolation inequality (\ref{ready-three}), we obtain
\begin{align}
\|\Phi\|_{H^1(\mathbb R^{n}\times (-R_3, R_3))}&\leq \exp(C_3 \lambda^{\frac{\beta_2}{2\beta_1} }) \|\phi\|^{\hat{\alpha}}_{L^2(\Omega\cap \Lambda_L)} \|\Phi\|^{1-\hat{\alpha}}_{H^1(\tilde{X}_{R_3})},
\end{align}
where we also use the fact $\frac{\beta_2}{2\beta_1} \ge \frac12$.
Since $\tilde{X}_{R_3}\subset \mathbb R^n\times (-R_3, R_3)$,  it follows that
\begin{align}
\|\Phi\|_{H^1(\mathbb R^{n}\times (-R_3, R_3))} \leq \exp(\hat{\alpha}^{-1} C_3 \lambda^{\frac{\beta_2}{2\beta_1} }) \|\phi\|_{L^2(\Omega\cap \Lambda_L)}.
\end{align}
Recall that $\hat{\alpha}^{-1} \approx \alpha_1^{-1} \approx |\log \delta| \approx |\log \gamma| \lambda^{\frac{\sigma}{\beta_1}}$. It follows that
\begin{align}
\|\Phi\|_{H^1(\mathbb R^{n}\times (-R_3, R_3))} &\le \exp(C |\log \gamma| \lambda^{\frac{\sigma}{\beta_1} + \frac{\beta_2}{2\beta_1} }) \|\phi\|_{L^2(\Omega\cap \Lambda_L)} \nonumber\\ &= (\frac{1}{\gamma})^{C\lambda^{\frac{\sigma}{\beta_1}+\frac{\beta_2}{2\beta_1}}}  \|\phi\|_{L^2(\Omega\cap \Lambda_L)}.
\label{fina-3}
\end{align}

Finally, using the lower bound in (\ref{com-phi}) with $\rho = R_3$, we obtain
\begin{align}
\|\phi\|_{L^2(\mathbb R^n)}\leq (\frac{1}{2R_3})^{\frac{1}{2}}\|\Phi\|_{H^1(\mathbb R^{n}\times (-R_3, R_3))} \le (\frac{1}{\gamma})^{C\lambda^{\frac{\sigma}{\beta_1}+\frac{\beta_2}{2\beta_1}}}  \|\phi\|_{L^2(\Omega\cap \Lambda_L)}.
\label{fina-2}
\end{align}
This completes the proof of Theorem \ref{th1}.
\end{proof}

\section{Appendix}
In this Appendix, we prove the quantitative global Carleman estimate in Proposition \ref{Pro4} and a lemma concerning the geometric properties on the introduced sets. Interested readers may refer to the survey \cite{K} and \cite{LL}
for more exhaustive literature for local and global Carleman estimates.
We will use the integration by parts repeatedly to get the desired estimate. Recall that
the weight function we are using is $$\psi(\hat{r}(y))= e^{-s\hat{r}(y)}, $$
where $\hat{r}(y)=|y-b|$, where $b=(0, \cdots, 0, -b_{n+1})$ with small $b_{n+1} = \frac{\delta}{100}$ and $\delta \in (0,\frac12)$. For convenience,
let $h(y)=-\hat{r}(y)=-|y-b|$.
Note that
\begin{align}
|D h(y)|=1 \quad \mbox{and} \quad |D^2 h(y)|\leq \frac{C}{\delta}
\label{hhnew}
\end{align} for $y\in \mathbb B^+_\delta$.
Next we present the proof of  Proposition \ref{Pro4}.

\begin{proof}[Proof of  Proposition \ref{Pro4}]
We first prove the desired estimate for the case $V = V_1\in W^{1,\infty}(\mathbb{B}_{\delta}^+)$.

Choose
\begin{equation} w(y)= e^{\tau \psi(y)} u(y).
\label{wvv}\end{equation} 
Since $u(y)\in C_0^\infty (\mathbb B^+_{\delta}\cup\{y_{n+1}=0\})$, then $w(y)\in C_0^\infty (\mathbb B^+_{{\delta}}\cup\{y_{n+1}=0\})$. We introduce  an $(n+1)$-dimensional second order elliptic operator $$P_0=-\triangle +V(y).$$
Define the conjugate operator as
$$ P_\tau w= e^{\tau\psi(y)} P_0( e^{-\tau\psi(y)}w).  $$
Direct calculations show that
\begin{align}
P_\tau w=&-\triangle w+2\tau D \psi\cdot  D w -\tau^2 |D \psi |^2  w+ \tau\triangle \psi w +V(y) w \nonumber \\
=&-\triangle w+2\tau s\psi D h\cdot  D w-\tau^2 s^2\psi^2  w+\tau\psi a(y,s)w+V(y) w\nonumber ,
\end{align}
where
\begin{equation} a(y,s)=s^2 +s \triangle h.
\label{axs}
\end{equation}
We split the expression $P_\tau w $ into the sum of two expressions $P_1 w$ and $P_2 w$, where
$$ P_1 w=-\triangle w-\tau^2 s^2\psi^2  w+V(y) w, $$
$$ P_2 w=2\tau s\psi D h \cdot D w.  $$
Then
\begin{equation}
P_\tau w= P_1 w+ P_2 w+ \tau\psi a(y,s)w.
\end{equation}
We compute the $L^2$ norm of $P_\tau w$. By the triangle inequality, we have
\begin{align}
\|P_\tau w\|^2&=\|P_1 w+P_2 w+\tau\psi a(y,s)w\|^2 \nonumber \\
&\geq \frac12 \|P_1 w\|^2+ \frac12 \|P_2 w\|^2+ \langle P_1 w, \  P_2 w\rangle -\|\tau\psi a(y,s)w\|^2.
\label{keyinner}
\end{align}
Later on, we will absorb the term $\|\tau\psi a(y,s)w\|^2$.  

Now we are going to derive a lower bound for the inner product in (\ref{keyinner}).
Let us write
\begin{equation}
\langle P_1 w, \  P_2 w\rangle=\sum^3_{k=1} I_k,
\label{right}
\end{equation}
where
\begin{align} &I_1= \langle -\triangle w, \  2\tau s\psi D h \cdot D w    \rangle, \nonumber \\
 &I_2=  \langle -\tau^2 s^2\psi^2  w, \      2\tau s\psi D h  \cdot D w         \rangle,  \nonumber \\
& I_3= \langle V(y)w, \      2\tau s\psi D h \cdot  D w\rangle,   
  \end{align}
We will estimate each term on the right-hand side of (\ref{right}). Recall that $w(y)\in C_0^\infty (\mathbb B^+_{{\delta}}\cup\{y_{n+1}=0\})$ and $h(x)$ satisfies (\ref{hhnew}). Performing the integration by parts shows that
\begin{align}
I_1=&2\tau s^2 \int_{\mathbb B_{{\delta}}^+} \psi (Dh\cdot D w)^2\,dy+2\tau s \int_{\mathbb B_{{\delta}}^+}  \psi  D_j w D_{ij} h D_i w\,dy \nonumber \\ &\quad +2\tau s\int_{\mathbb B_{{\delta}}^+}\psi D_j w D_{ij} w D_i h \,dy-2\tau s\int_{\partial \mathbb B_{{\delta}}^+\cap \{y_{n+1}=0\}} \psi  D h \cdot D w \frac{\partial w} {\partial y_{n+1}} \,dS\nonumber \\
=&I_1^{1}+I_1^2+ I_1^3+I_1^4,
\label{lala}
\end{align}
where we have adopted the convention $D_{ij} w = D_i D_j w$ and the Einstein notation so that the repeated index is summed. By the integration by parts, the third term $I_1^3$ can be computed as
\begin{align}
I_1^3&=\tau s \int_{\mathbb B_{{\delta}}^+}\psi D_i h D_i (|D w|^2)\,dy \nonumber  \\
&=-\tau s\int_{\mathbb B_{{\delta}}^+} \psi \triangle h |D w|^2 \,dy-\tau s^2 \int_{\mathbb B_{{\delta}}^+}\psi  |D w|^2\,dy \nonumber  \\
&\quad +\tau s\int_{\partial \mathbb B_{{\delta}}^+\cap \{y_{n+1}=0\}}\psi  |D w|^2 \frac{\partial h} {\partial y_{n+1}} dS.
\label{lala2}
\end{align}
Combining (\ref{hhnew}), (\ref{lala}) and (\ref{lala2}), we can estimate $I_1$ from below
\begin{align}
I_1 \geq -\tau s^2 \int_{\mathbb B_{{\delta}}^+}\psi  | Dw|^2 \, dy- \frac{C\tau s}{\delta}  \int_{\mathbb B_{{\delta}}^+}\psi |D w|^2\, dy-C\tau s\int_{\partial \mathbb B_{{\delta}}^+\cap\{y_{n+1} = 0\}}\psi |\frac{\partial w}{\partial y_{n+1}}|^2\, dy,
\label{II1}
\end{align}
where we have also used the fact $D_{i} w  =0 $ on $\partial \mathbb B_{{\delta}}^+ \cap \{ y_{n+1} = 0\}$ for $i=1,2,\cdots, n$. Choosing $s = \frac{C_2}{\delta}$ for sufficiently large $C_2$ (dependig only on $n$), we have
\begin{equation}
    I_1 \ge -\frac65 \tau s^2 \int_{\mathbb B_{{\delta}}^+}\psi  | Dw|^2 \, dy -C\tau s\int_{\partial \mathbb B_{{\delta}}^+\cap\{y_{n+1} = 0\}}\psi |\frac{\partial w}{\partial y_{n+1}}|^2\, dy.
\end{equation}

Now we compute the term $I_2$ using integration by parts. Since $w(y)\in C_0^\infty (\mathbb B^+_{\delta}\cup\{y_{n+1}=0\})$, we have
\begin{align}
I_2&=-\tau^3s^3 \int_{\mathbb B_{{\delta}}^+}\psi^3   D h \cdot D w^2 \, dy \nonumber  \\
&=3 \tau^3 s^4 \int_{\mathbb B_{{\delta}}^+}\psi^3   w^2 \, dy +\tau^3 s^3 \int_{\mathbb B_{{\delta}}^+} \psi^3  \triangle h w^2 \, dy \nonumber  \\
&\geq 3 \tau^3 s^4 \int_{\mathbb B_{{\delta}}^+}\psi^3   w^2 \, dy- \frac{C\tau^3 s^3}{\delta}  \int_{\mathbb B_{{\delta}}^+  }\psi^3   w^2 \, dy.
\end{align}
Again, choosing $s = \frac{C_2}{\delta}$ with $C_2$ large enough, we deduce that
\begin{align}
I_2\geq \frac{14}{5} \tau^3 s^4 \int_{\mathbb B_{{\delta}}^+}\psi^3  w^2 \, dy.
\label{II2}
\end{align}

We proceed to estimate the term $I_3$. It is here that we take advantage of the gradient information of the potential. Using integration by parts shows that
\begin{align}
I_3 &=\tau s \int_{ \mathbb B_{{\delta}}^+} V(y) \psi  D h \cdot D w^2 \,dy \nonumber\\
& =-\tau s \int_{ \mathbb B_{{\delta}}^+} \psi  D V(y)\cdot D h  w^2  \,dy-\tau s^2 \int_{ \mathbb B_{{\delta}}^+} V(y) \psi
 w^2\,dy\nonumber\\
&\quad -\tau s \int_{ \mathbb B_{{\delta}}^+} V(y)\psi  \triangle h w^2\,dy.
\end{align}
Thus, we have
\begin{align}
I_3\geq -C\tau s^2 \|V\|_{W^{1,\infty}(\mathbb{B}_{\delta}^+)}\int_{ \mathbb B_{{\delta}}^+}\psi w^2\,dy .
\label{II4}
\end{align}
Now, assume $\tau > C_1 \|V\|_{W^{1,\infty}(\mathbb{B}_{\delta}^+)}^{\frac12} $ with sufficiently large $C_1>0$. Then
\begin{equation}
    I_3 \ge -\frac15 \tau^3 s^2 \int_{ \mathbb B_{{\delta}}^+}\psi^3  w^2\,dy,
\end{equation}
where we also used the observation $\psi \approx 1$ in $\mathbb B_{{\delta}}^+$ (this is because $s\approx \delta^{-1}$, $\hat{r} \approx \delta$ in $\mathbb B_{{\delta}}^+$ and $\psi = e^{-s\hat{r}}$).

Combining all the estimates on each $I_k$, we arrive at
\begin{align}\label{III}
\langle P_1 w, \  P_2 w\rangle = \sum^3_{k=1}I_k&\geq \frac{13}{5}\tau^3 s^4 \int_{\mathbb B_{{\delta}}^+}\psi^3  w^2 \, dy 
-\frac{6}{5} \tau s^2 \int_{\mathbb B_{{\delta}}^+}\psi  | Dw|^2 \, dy \\
& \quad -C\tau s
\int_{\partial \mathbb B_{{\delta}}^+\cap \{y_{n+1}=0\}}\psi |\frac{\partial w}{\partial y_{n+1}}|^2 \, dS,  \nonumber 
\end{align}
under the assumptions $s = \frac{C_2}{\delta}$ and $\tau > C_1 \|V\|_{W^{1,\infty}(\mathbb{B}_{\delta}^+)}^{\frac12}$ for sufficiently large $C_1, C_2 >0$.

We want to control the gradient term on the second term on the right-hand side of (\ref{III}). To this end, we consider the following inner product
\begin{equation}
\langle P_1 w, \ \tau s^2\psi w\rangle =\sum^3_{k=1} L_k,
\label{product}
\end{equation}
where
\begin{align}
&L_1=\langle -\triangle w, \ \tau s^2 \psi  w\rangle,\nonumber\\
&L_2=\langle -\tau^2 s^2\psi^2  w, \tau s^2\psi w\rangle=-\tau^3 s^4\int_{ \mathbb B_{{\delta}}^+}\psi^3  w^2\,dy,
\label{LL2}
\end{align}
and
\begin{align}
L_3&=\langle V(y) w,\ \tau s^2 \psi w  \rangle \ge 0.
\label{LL4}
\end{align}
We want to find out a lower bound estimate for $L_1$ to include the gradient term. Recall that $w(y)\in C_0^\infty (\mathbb B^+_{\delta}\cup\{y_{n+1}=0\})$. It follows from the integration by parts and Cauchy-Schwartz inequality that
\begin{align}
L_1&=\tau s^2 \int_{ \mathbb B_{{\delta}}^+}\psi | D w|^2\, dy+\tau s^3  \int_{ \mathbb B_{{\delta}}^+}  \psi Dh \cdot  D w w\,dy \nonumber \\
&\qquad -\tau s^2\int_{\partial \mathbb B_{{\delta}}^+\cap\{y_{n+1}=0\}}\psi  |\frac{\partial w}{\partial y_{n+1}}|^2 d S \nonumber \\
&\geq  \frac{4}{5}\tau s^2 \int_{ \mathbb B_{{\delta}}^+}\psi  |D w|^2 \, dy-C\tau s^{4} \int_{ \mathbb B_{{\delta}}^+}\psi w^2 \, dy \nonumber \\
&\quad -\tau s^2\int_{\partial \mathbb B_{{\delta}}^+\cap\{y_{n+1}=0\}}\psi |\frac{\partial w}{\partial y_{n+1}}|^2 d S.
\label{LL1}
\end{align}
Taking (\ref{product}) -- (\ref{LL1})  and $\tau> C_1$ into account gives
\begin{align}
\langle P_1 w, \ 2 \tau s^2 \psi  w\rangle &\geq \frac{8}{5} \int_{ \mathbb B_{{\delta}}^+}\psi  |Dw|^2 \, dy - \frac{11}{5}\tau^3 s^4 \int_{ \mathbb B_{{\delta}}^+}\psi^3  w^2\,dy \nonumber \\
&\quad -C\tau s^2 \int_{\partial \mathbb B_{{\delta}}^+\cap\{y_{n+1}=0\}}\psi |\frac{\partial w}{\partial y_{n+1}}|^2\, dS.
\label{LLLii}
\end{align}
Since
\begin{align}
\frac12 \|P_1 w\|^2 +2 \|\tau s^2\psi w\|^2 \geq \langle P_1 w,  2\tau s^2 \psi w\rangle,
\end{align}
from (\ref{keyinner}), we obtain
\begin{align}
\|P_\tau w\|^2 +2 \|\tau s^2\psi w\|^2 &\geq \langle P_1 w,  2\tau s^2 \psi w\rangle + \langle P_1 w,  P_2 w\rangle -\|\tau\psi a(y,s)w\|^2 \nonumber \\
& \ge \frac{2}{5}\tau^3 s^4 \int_{ \mathbb B_{{\delta}}^+}\psi^3  w^2\,dy + \frac{2}{5} \int_{ \mathbb B_{{\delta}}^+}\psi  |Dw|^2 \, dy \nonumber \\
& \quad - C\tau s^2 \int_{\partial \mathbb B_{{\delta}}^+\cap\{y_{n+1}=0\}}\psi |\frac{\partial w}{\partial y_{n+1}}|^2\, dS -\|\tau\psi a(y,s)w\|^2.
\label{dom-af}
\end{align}

Next, using \eqref{axs} and \eqref{hhnew}, we obtain
\begin{equation}\label{axs-w}
    2 \|\tau s^2\psi w\|^2 + \|\tau\psi a(y,s)w\|^2 \le C\tau^2 s^4 \int_{\mathbb{B}_{\delta}^+ } \psi^2 w^2 \le \frac15 \tau^3 s^4 \int_{\mathbb{B}_{\delta}^+ } \psi^3 w^2,
\end{equation}
where in the second inequality, we assume $\tau > C_1$ for sufficiently large $C_1$ and use the fact $\psi \approx  1$.
Combining \eqref{dom-af} and \eqref{axs-w}, we get
\begin{align}\label{main.Ptau}
\|P_\tau w\|^2&+ C \tau s^2
\int_{\partial \mathbb B_{{\delta}}^+\cap\{y_{n+1}=0\}}\psi |\frac{\partial w}{\partial y_{n+1}}|^2 \, dS \nonumber \\
 &\geq \frac15 \tau^3 s^4 \int_{\mathbb B_{{\delta}}^+}\psi^3  w^2 \, dy + \frac25 \tau s^2\int_{ \mathbb B_{{\delta}}^+} \psi|D w|^2\,dy.
\end{align}
Since $w = e^{\tau \psi} u$, the triangle inequality and \eqref{hhnew} imply
\begin{equation}
    |D w|^2 \ge \frac12 e^{2\tau \psi} |D u|^2 - \tau^2 s^2 \psi^2 e^{2\tau \psi} u^2.
\end{equation}
This, together with \eqref{main.Ptau}, the definition of $P_\tau$ and the fact $u(y)\in C_0^\infty (\mathbb B^+_{\delta}\cup\{y_{n+1}=0\})$, gives
\begin{align}
\| e^{\tau \psi} P_0 u \|^2&+ C \tau s^2
\int_{\partial \mathbb B_{{\delta}}^+\cap\{y_{n+1}=0\}}\psi  e^{2\tau \psi}|\frac{\partial u}{\partial y_{n+1}}|^2 \, dS \nonumber \\
 &\geq \frac{1}{10} \tau^3 s^4 \int_{\mathbb B_{{\delta}}^+}\psi^3  e^{2\tau \psi} u^2 \, dy + \frac{1}{5} \tau s^2\int_{ \mathbb B_{{\delta}}^+} \psi e^{2\tau \psi}|D u|^2\,dy.
 \label{CarlemanP0}
\end{align}
Thus, we have proved the desired estimate for the case $V \in W^{1,\infty}(\mathbb{B}_\delta^+)$. Recall that here we assume $s = \frac{C_2}{\delta}$ and $\tau > C_1(1+ \|V\|_{W^{1,\infty}(\mathbb{B}_{\delta}^+)}^{\frac12})$ for sufficiently large $C_1, C_2>0$.

Finally, we cosider the general case $V = V_1 + V_2$ with $V_1 \in W^{1,\infty}(\mathbb{B}_\delta^+)$ and $V_2 \in L^\infty(\mathbb{B}_\delta^+)$. Here is the small trick to get the expected exponent for $V_2$. First, the estimate \eqref{CarlemanP0} holds with $P_0 = (-\triangle + V_1)$ and 
$\tau > C_1(1+\norm{V_1}_{W^{1,\infty}}^{\frac12} ).$
To insert $V_2$, using $V_1 = V - V_2$ and the triangle inequality, we have
\begin{align}
    \| e^{\tau \psi}(- \triangle u +V_1 u)\|^2 \le 2\| e^{\tau \psi}(- \triangle u +Vu)\|^2 + 2 \norm{V_2}^2_{L^\infty(\mathbb{B}_{\delta}^+) } \| e^{\tau \psi} u\|^2 .
\end{align}
Note that the last term on the right-hand side can be absorbed to the right-hand side of \eqref{CarlemanP0}, provided $\tau > C_1(1+\norm{V_1}_{W^{1,\infty}(\mathbb{B}_{\delta}^+)}^{\frac12} +\norm{V_2}_{L^\infty(\mathbb{B}_{\delta}^+)}^{\frac23})$ for sufficiently large $C_1$. This completes the proof of Proposition \ref{Pro4}.
\end{proof}

In the following lemma, we will show the conclusions (\ref{state-2}) and (\ref{three-4}), which were also proved 
 in \cite[Lemma 3.3]{NTTV}.
\begin{lemma}
Let $L\geq 5$ be an odd integer and $\delta\in (0, \frac{1}{2})$. Let $\{ z_i\}$ be a $(1, \delta)$-equidistributed sequence in $\Lambda_L$ and $Q_L = \Lambda_L \cap \Z^n$.
\begin{itemize}
    \item[(i)] It holds $X_1\subset \cup_{i\in Q_L} S_2(z_i)$.

    \item[(ii)] Let $f \in H^1_{\rm loc}(\R^n)$. There exists $\gamma_n$ depending on $n$ such that
\begin{align*}
\sum_{i\in Q_L} \|f\|_{H^1(S_3(z_i))}\leq \gamma_n \|f\|_{H^1(\cup_{i\in Q_L}S_3(z_i))}.
\end{align*}
\end{itemize}
\label{lemma-app}
\end{lemma}
\begin{proof}
(i). We first show that $[-\frac{1}{2}, \frac{1}{2}]^n\times [-1, 1]$ can be covered by the sets $S_2(z_i)$. Let us choose
$i_1=(-1, 0, \cdots, 0)$ and $i_2=(1, 0, \cdots, 0)$, $i_1,i_2\in Q_L$. We claim that
\begin{align*}
[-\frac{1}{2}, \frac{1}{2}]^n\times [-1, 1]\subset  S_2(z_{i_1})\cup S_2(z_{i_2}).
\label{dist-1}
\end{align*}
To this end, recall that $S_2=\mathbb B_{3\sqrt{n}}\backslash \overline{\mathbb B_{\frac{1}{2}} }$.
Since $\{ z_{i}\}$ is a $(1, \delta)$-equidistributed sequence, then
$z_{i_1}\in (-\frac{3}{2}+\delta, -\frac{1}{2}-\delta)\times (-\frac{1}{2}+\delta, \frac{1}{2}-\delta)^{n-1}$ and
$z_{i_2}\in (\frac{1}{2}+\delta, \frac{3}{2}-\delta)\times (-\frac{1}{2}+\delta, \frac{1}{2}-\delta)^{n-1}$. It is obvious that
\begin{equation*}
    [-\frac{1}{2}, \frac{1}{2}]^n\times [-1, 1]\subset \mathbb{B}_{3\sqrt{n}}(z_{i_1}) \cup \mathbb{B}_{3\sqrt{n}}(z_{i_2}) \quad \text{and} \quad \overline{ \mathbb{B}_{\frac12}(z_{i_1})} \cap \overline{ \mathbb{B}_{\frac12}(z_{i_2})} = \emptyset.
\end{equation*}
It follows that
\begin{equation*}
\begin{aligned}
    S_2(z_{i_1})\cup S_2(z_{i_2}) & = \Big( \mathbb B_{3\sqrt{n}}(z_{i_1})\backslash \overline{\mathbb B_{\frac12}(z_{i_1}) } \Big) \cup \Big( \mathbb B_{3\sqrt{n}}(z_{i_2})\backslash \overline{\mathbb B_{\frac12}(z_{i_2}) } \Big) \\
    & = \Big( \mathbb B_{3\sqrt{n}}(z_{i_1}) \cup \mathbb B_{3\sqrt{n}}(z_{i_2}) \Big) \setminus \Big( \overline{ \mathbb{B}_{\frac12}(z_{i_1})} \cap \overline{ \mathbb{B}_{\frac12}(z_{i_2})} \Big) \\
    & = \mathbb B_{3\sqrt{n}}(z_{i_1}) \cup \mathbb B_{3\sqrt{n}}(z_{i_2}) \\
    & \supset [-\frac{1}{2}, \frac{1}{2}]^n\times [-1, 1].
\end{aligned} 
\end{equation*}
This proves the claim.

The set $X_1$ consists of the cell $(i+[-\frac{1}{2}, \frac{1}{2}]^n )\times [-1, 1]$ for $i\in Q_L$ for some $L\geq 5$. We can apply the same argument as before to cover each $(i+[-\frac{1}{2}, \frac{1}{2}]^n ) \times [-1, 1]$ by two neighboring sets $S_2(z_i)$.

(ii) We observe that, each $x\in \cup_{i\in Q_L} S_3(z_i)$ can be covered by $(2R_3+2)^n$ number of $S_3(z_i)$ for $i\in Q_L$. Thus,
\begin{align*}
\sum_{i\in Q_L} \chi_{S_3(z_i)}(x)\leq (2R_3+2)^n \chi_{\cup_{i\in Q_L} S_3(z_i) }(x).
\end{align*}
Hence  we have
\begin{align*}
\sum_{i\in Q_L} \|f\|^2_{H^1(S_3(z_i))}&\leq \sum_{j\in Q_L} \int_{S_3(z_i)} |D f|^2+ |f|^2 dx \nonumber \\
&\leq 
\gamma_n \|f\|^2_{H^1(\cup_{i\in Q_L} S_3(z_i))},
\end{align*}
where $\gamma_n=(2R_3+2)^n$.
\end{proof}


\begin{thebibliography}{CL}

\bibitem[A82]{A}S. Agmon, {\em Lectures on Exponential Decay of Solutions of Second-Order Elliptic Equations: Bounds on Eigenfunctions of $N$-Body Schr\"{o}dinger Operations}, Math. Notes, Vol. 29. Princeton University Press, 1982.

\bibitem[ARRV09]{ARRV} G. Alessandrini, L. Rondi, E. Rosset and S. Vessella,
{\em The stability for the Cauchy problem for elliptic equations},
Inverse Problems,  25 (2009), no. 12, 123004, 47 pp.

\bibitem[AS22]{AS} P. Alphonse and A. Seelmann, {\em Quantitative spectral inequalities for the anisotropic Shubin operators and the Grushin operator}, arXiv:2212.10842 (2022).

\bibitem[BJP21]{BJP} K. Beauchard, P. Jaming, and K. Pravda-Starov, {\em Spectral estimates for finite combinations
of Hermite functions and null-controllability of hypoelliptic quadratic equations}, Studia Math., 260 (2021), no.1, 1--43.



\bibitem[BS91]{BS} F.A. Berezin and M.A. Subin, {\em the Sch\"odinger equation}, Kluwer, Dordrecht, 1991.




\bibitem[DSV22a]{DSV1} A.   Dicke, A. Seelmann, and I. Veseli\'c, {\em Uncertainty principle for Hermite functions and null-controllability with sensor sets of decaying density},  arXiv:2201.11703 (2022).

\bibitem[DSV22b]{DSV} A. Dicke,  A. Seelmann and I. Veseli\'c, {\em Spectral inequality with sensor sets of decaying density for Schr\"odinger operators with power growth potentials},  arXiv:2206.08682 (2022).




\bibitem[DZ16]{DZ}S. Dyatlov and J.Zahl, {\em Spectral gaps, additive energy, and a fractal uncertainty
principle}, Geom. Funct. Anal., 26 (2016), 1011-1094.

\bibitem[DF88]{DF}
H. Donnelly and C. Fefferman,
 {\em Nodal sets of eigenfunctions on Riemannian manifolds},
 {Invent. Math.}, 93 (1988), no. 1, 161-183.

\bibitem[GY12]{GY}J. Gagelman and H. Yserentant, {\em A spectral method for Schr\"odinger equations with smooth confinement potentials},  Numerische Mathematik, 122, no. 2 (2012): 383-398.


\bibitem[JL96]{JL} D. Jerison and G. Lebeau, {\em Nodal sets of sums of eigenfunctins, Harmonic analysis and partial differential equations (Chicago, IL, 1996)}, 223-239,
Chicago Lectures in Math., Uniw. Chicago Press, Chicago, IL, 1999.




\bibitem[K07]{K}
C.  Kenig,
 {\em Some recent applications of unique continuation,
 In  Recent developments in nonlinear partial differential
  equations},  25-56, Contemp. Math, 439, Amer. Math.
  Soc., Providence, RI, 2007.

\bibitem[KSW15]{KSW} C. E. Kenig, L. Silvestre and J.-N. Wang, {\em On Landis Conjecture in the Plane}, Communications in Partial Differential Equations, 40 (2015), no.4, 766-789.



\bibitem[KL88]{KL} V. A. Kondratiev and E. M. Landis, {\em Qualitative theory of second order linear partial
differential equations, Partial differential equations 3}, Itogi Nauki i Tekhniki,
Ser. Sovrem. Probl. Mat. Fund. Napr., 32, VINITI, Moscow, 1988, 99-215.

\bibitem[K01]{K}O. Kovrijkine, Some results related to the Logvinenko-Sereda Theorem, Proc. Amer. Math. Soc. 129,
(2001), no. 10, 3037-3047. 

\bibitem[Lau11]{Lau} B. Laurent, {\em Quantitative uniqueness for Schr\"odinger operator}. Indiana Univ. Math. J. 61 (2012), no. 4, 1565-1580.

\bibitem[Lin91]{Lin} F.-H. Lin, {\em Nodal sets of solutions of elliptic equations of elliptic and parabolic equations}, Comm. Pure Appl Math.,
44 (1991), 287-308.

\bibitem[LR95]{LR} G. Lebeau and L. Robbiano, {\em Contr\^ole exacte de l'\'equation de la chaleur}, Comm. Partial Differential Equations, 20 (1995), 335-356.

\bibitem[LZ98]{LZ} G. Lebeau and E. Zuazua, {\em Null-controllability of a system of linear thermoelasticity}, Arch. Rational Mech. Anal. 141 (1998), 297-329.



\bibitem[LL12]{LL}J. Le Rousseau and G. Lebeau, {\em On Carleman estimates for elliptic and parabolic operators. Applications to unique continuation and control of parabolic equations}, ESAIM Control Optim. Calc. Var.,  18 (2012), no.3, 712-747.





\bibitem[LMNN20]{LMNN} A. Logunov, E. Malinnikova, N. Nadirashvili, and F. Nazarov, {\em The Landis conjecture on exponential decay}, arXiv:2007.07034 (2020).

\bibitem[MP20]{MP} J. Martin and K. Pravda-Starov, {\em Spectral inequalities for combinations of Hermite functions and null-controllability for evolution equations enjoying Gelfand-Shilov smoothing effects}, arXiv:2007.08169 (2020).

\bibitem[NTTV18a]{NTTV} I. Naki\'c, M T\"aufer, M. Tautenhahn, and I. Veseli\'c, {\em Scale-free unique continuation principle for spectral projectors, eigenvalue-lifting and Wegner estimates for random Schr\"odinger operators}, Analysis \& PDE 11 (2018), no. 4, 1049-1081.



\bibitem[NTTV20a]{NTTV1}
I. Naki\'c, M. T\"aufer, M. Tautenhahn, and I. Veseli\'c, {\em Sharp estimates and homogenization of the control cost of the heat equation on large domains}, ESAIM Control Optim. Calc. Var., 26 (2020), no. 54, 26 pages.

\bibitem[NTTV20b]{NTTV2}
I. Naki\'c, M. T\"aufer, M. Tautenhahn, and I. Veseli\'c, {\em Unique continuation and lifting of spectral band edges of Schr\"{o}dinger operators on unbounded domains}, J. Spectr.
Theory, 10 (2020), no. 3, 843-885.

\bibitem[RV13]{RV} C. Rojas-Molina and I Veseli\'c, {\em Scale-free unique continuation estimates and applications to random Schr\"odinger operators}, Comm. Math. Phys. 320 (2013), no. 1, 245-274.



\bibitem[Zh15]{Zh}J. Zhu, {\em Doubling property and vanishing order of Steklov
eigenfunctions}, Comm. Partial Differential Equations, 40 (2015), no.
8, 1498-1520.

\bibitem[Zh16]{Zh1}
J. Zhu, {\em Quantitative uniqueness of elliptic equations}, Amer. J. Math.
  {138} (2016), no.~3, 733--762.

\bibitem[Zh22]{Zh2}J. Zhu, {\em Boundary doubling inequality and nodal sets of Robin and Neumann eigenfunctions}, Potential Anal., to appear.

\end{thebibliography}
\end{document}